\documentclass[reqno]{amsart}
\usepackage{amssymb,amsthm,amsfonts,amstext}
\usepackage{amsmath}
\usepackage{zito}
\usepackage{newcent}       % selects Times Roman as basic font
\usepackage{helvet}         % selects Helvetica as sans-serif font
\usepackage{courier}        % selects Courier as typewriter font

\usepackage{url}

\makeatletter \@addtoreset{equation}{section} \makeatother

\renewcommand\thefigure{\thesection.\@arabic\c@figure}
\renewcommand\thetable{\thesection.\@arabic\c@table}
\newtheorem{theorem}{Theorem}[section]
\newtheorem{lemma}[theorem]{Lemma}
\newtheorem{proposition}[theorem]{Proposition}
\newtheorem{corollary}[theorem]{Corollary}

\newcommand{\mc}[1]{{\mathcal #1}}

\newcommand{\bb}[1]{{\mathbb #1}}

\newcommand{\<}{\langle}
\renewcommand{\>}{\rangle}

\DeclareMathOperator{\supp}{supp}

\title{Universality of KPZ equation}

\author{Milton Jara}
\address{CEREMADE \\ Universit\'e Paris-Dauphine\\
Place du Mar\'echal de Lattre de Tassigny\\
Paris CEDEX 75775\\
France\\
}

\email{jara@ceremade.dauphine.fr}

\author{Patr\'{\i}cia Gon\c{c}alves}
\address{CMAT, Centro de Matemática da Universidade do Minho, Campus de Gualtar, 4710-057 Braga, Portugal}
\email{patg@math.uminho.pt}


\begin{document}

\begin{abstract}
We introduce the notion of energy solutions of the KPZ equation. Under minimal assumptions, we prove that the density fluctuations of one-dimensional, weakly asymmetric, conservative particle systems with respect to the stationary states are given by energy solutions of the KPZ equation. As a consequence, we prove that the Cole-Hofp solutions are also energy solutions of the KPZ equation.
\end{abstract}

\maketitle

\section{Introduction}

In a seminal paper \cite{KPZ}, Kardar, Parisi and Zhang proposed a phenomenological model for  the stochastic evolution of the profile of a growing interface $h_t(x)$. The so-called KPZ equation has the following form in one dimension:
\[
\partial_t h = D \Delta h + a \big(\nabla h\big)^2 + \sigma {\mc W}_t,
\]
where $\mc W_t$ is a space-time white noise and the constants $D$, $a$, $\sigma$ are related to some thermodynamic properties of the interface. The quantity $h_t(x)$ represents the {\em height} of the interface at the point $x \in \bb R$. From a mathematical point of view, this equation is ill-posed, since the solutions are expected to look locally like a Brownian motion, and in this case the nonlinear term does not make sense, at least not in a classical sense. 
This equation can be solved at a formal level using the Cole-Hopf transformation $Z_t(x) = \exp\{a/D h_t(x)\}$, which transforms this equation into the stochastic heat equation
\[
\partial_t Z = D \Delta Z + a \sigma /D Z {\mc W}_t.
\]
This equation is now linear, and mild solutions can be easily constructed. We will call these solutions {\em Cole-Hopf} solutions. It is widely believed that the physically relevant solutions of the KPZ equation are the Cole-Hopf solutions. However, the KPZ equation has been so resistant to any attempt to mathematical rigor, that up to now it has not even been proved that Cole-Hopf solutions satisfy the KPZ equation in any meaningful sense. Some interpretations that allow rigorous results are proved to give non-physical solutions \cite{Chan}. Up to our knowledge, the best effort in this direction corresponds to the work of Bertini and Giacomin \cite{BG}. In that work, the authors prove two results. First they prove that the Cole-Hopf solutions can be obtained as the limit of a sequence of mollified versions of the KPZ equation. And secondly, they proved that the Cole-Hopf solution appears as the scaling limit of the fluctuations of the current for the weakly asymmetric simple exclusion process (WASEP), giving mathematical support to the physical relevance of the Cole-Hopf solution.

In dimension $d=1$, a conservative version of the KPZ equation can be obtained by defining $\mc Y_t = \nabla h_t$:
\[
\partial_t \mathcal Y_t = D\Delta \mathcal Y_t + a \nabla \mc Y_t^2 + \sigma \nabla {\mc W}_t.
\]
This equation has (always at a formal level!) a spatial white noise as an invariant solution. In this case is even clearer that some procedure is needed in order to define $\mc Y_t^2$ in a proper way. 

Since the groundbreaking works of Johansson \cite{Joh} and Baik, Deift and Johansson \cite{BDJ}, a new approach to the analysis of the KPZ equation has emerged. The general strategy is to describe various functionals of one-dimensional asymmetric, conservative systems in terms of determinantal formulas. These determinantal formulas turn out to be related to different scaling limits appearing in random matrix theory. We refer to the expository work \cite{FS} for further references and more detailed comments, and to \cite{SS}, \cite{ACQ} for recent advances.

Another approach to the analysis of fluctuations of one-dimensional conservative systems was proposed by Balazs and Seppalainen \cite{BS}. They call their approach {\em microscopic concavity/convexity}  and it is exploited in \cite{BQS} in order to prove that the Cole-Hopf solution of the KPZ equation has the scaling exponents predicted by physicists. 

The main drawback of all these approaches is the lack of robustness. The microscopic Cole-Hopf transformation used in \cite{BG} works only for the WASEP, and any other interaction different from the exclusion principle rules out this approach. Determinantal processes appear in a natural way for particle systems that can be described  by {\em non-intersecting paths}. The microscopic concavity/convexity property can be defined only for attractive systems, and it has been proved only for the asymmetric simple exclusion process (ASEP) and the totally asymmetric, nearest-neighbors zero-range process under very restrictive assumptions. Moreover, up to our knowledge all the approaches to the KPZ equation go through Bertini and Giacomin construction of the Cole-Hopf solutions (see however \cite{QV}).

It is widely believed in the physics community that the KPZ equation governs the large-scale properties  of one-dimensional, asymmetric, conservative systems in great generality. The microscopic details of each model should only appear through the values of the constants $D$, $a$ and $\sigma$. In this article we provide a new approach which is robust enough to apply for a wide family of one-dimensional, asymmetric systems. The payback of such a general approach comes at the level of the results: we are not able to prove the precise results of Bertini and Giacomin, and we can not recover the detailed results obtained by the random matrix theory approach. However, our approach is robust enough to give information about sample path properties of the solutions of the KPZ equation.

As a stochastic partial differential equation, the main problem with the KPZ equation is the definition of the square $\mc Y_t^2$. Spatial white noise is an invariant solution of the equation, and it is expected that physically relevant solutions look locally like white noise.

Our first contribution is the notion of {\em energy solutions} of the KPZ equation (see Section \ref{s1.3}). Various attempts to rigorously define a solution of the KPZ equation have been made. One possibility is to regularize the noise $\mc W_t$ and then to turn the regularization off. For the regularized problem, several properties, like well-posedness of the Cauchy problem and existence of invariant measures can be proved \cite{Sin}, \cite{EKMS}. However, the available results hold in a window which is still far from the white noise $\mc W_t$. Another possibility is to regularize the nonlinearity $\nabla \mc Y_t^2$ \cite{D-PD}. One more time, this procedure gives well-posedness in a window which is far from the KPZ equation. Yet another possibility corresponds to define the nonlinear term through a sort of Wick renormalization \cite{HLOUZ}, \cite{Ass2}. However, this procedure does not lead to solutions with the right scaling properties \cite{Chan}. Our notion of energy solutions is strong enough to imply some regularity properties of the solutions which allow to justify some formal manipulations.

We introduce the notion of energy solutions of the KPZ equation in order to state in a rigorous way our second contribution. Take a one-dimensional, weakly asymmetric conservative particle system and consider the rescaled space-time fluctuations of the density field $\mc Y_t^n$ (see the definition in Section \ref{s1.2}). The strength of the asymmetry is of order $1/\sqrt n$. For the speed-change simple exclusion process considered in \cite{FHU} and starting from an stationary distribution, we prove that any limit point of $\mc Y_t^n$ is an energy solution of the KPZ equation. The only ingredients needed in order to prove this result are a sharp estimate on the spectral gap of the dynamics of the particle system restricted to finite boxes (stated in Proposition \ref{p6}) and a strong form of the equivalence of ensembles for the stationary distribution (see Proposition \ref{p9}). Therefore, our approach works, modulo technical modifications, for any one-dimensional, weakly asymmetric conservative particle system satisfying these two properties. In particular, our approach is suitable to treat models like the zero-range process and Ginzburg-Landau model in dimension $d=1$, for which the methods mentioned before fail dramatically. Our approach also works for models with finite-range, non-nearest neighbor interactions with basically notational modifications.

We consider in this article speed-change exclusion processes satisfying the so-called {\em  gradient condition}. Notice that we need to know the invariant measures of the model in order to state the equivalence of ensembles. It has been proved \cite{Nag} for the speed-change exclusion process that the invariance of a Gibbs measure under the symmetric dynamics is preserved by introducing an asymmetry, if and only if the model satisfies the gradient condition. It is only at this point that we need the gradient condition. In particular, our approach also works for weakly asymmetric systems for which the invariant measures are known explicitly, even if the gradient condition is not satisfied.
In view of this discussion, we say that energy solutions of the KPZ equation are {\em universal}, in the sense that they arise as the scaling limit of the density in one-dimensional, weakly asymmetric conservative systems satisfying fairly general, minimal assumptions.

In order to prove this theorem, we introduce a new mathematical tool, which we call {\em second-order Boltzmann-Gibbs principle}. The usual Boltzmann-Gibbs principle, introduced in \cite{BR} and proved in \cite{D-MPSW} in our context, basically states that the space-time fluctuations of any field associated to a conservative model can be written as a linear functional of the density field $\mc Y_t^n$. Our second-order Boltzmann-Gibbs principle states that the first-order correction of this limit is given by a singular, quadratic functional of the density field. It has been proved that in dimension $d \geq 3$, this first order correction is given by a white noise \cite{CLO}. It is conjectured that this is also the case in dimension $d=2$ and in dimension $d=1$ if our first-order correction is null. 

The rest of this paper is organized as follows. In Section \ref{s1} we give precise definitions of the model considered here and we state the results proved in the rest of the article. In Section \ref{s1.1} we define the speed-change exclusion process and we state some of its basic properties. In Section \ref{s1.2} we give an overview of various scaling limits of the density of particles for the speed-change, simple exclusion process, and in particular we state Bertini and Giacomin's result. In Section \ref{s1.3} we give rigorous definitions of what we understand by weak and energy solutions of the KPZ equation and we state our main result. In few words, our main result states that the density fluctuation field is tight, and any limit point is an energy solution of the KPZ equation. We state two corollaries, the first one gives the H\"older exponent of the energy solutions of the KPZ equation. The second gives an answer of the following open problem in \cite{BG}: does the Cole-Hopf solution of the KPZ equation actually satisfies the KPZ equation in any meaningful way? Applying our main result to the WASEP, we prove that the stationary Cole-Hopf solution is an energy solution of the KPZ equation. In Section \ref{s1.4}, we define the stochastic growth model associated to the speed-change, simple exclusion process and we restate the definitions of Section \ref{s1.3} in terms of this growth model. We prove a central limit theorem for the current through a bond. For simplicity, we assume that the drift has zero average, but our results remain true if we look at the current across a characteristic. Starting from this result we prove a convergence result for the height fluctuations on this stochastic growth model.

In Section \ref{s3} we state and prove the second-order Boltzmann-Gibbs principle, which is is the main technical innovation of this article. In Section \ref{s3.1} we review Kipnis-Varadhan and spectral gap inequalities and the equivalence of ensembles, and we state Propositions \ref{p7} and \ref{p9}. We point out here that the equivalence of ensembles we need is of second-order and in consequence it is finer than the result usually found in the literature. For the reader's convenience we give in Appendix \ref{A1} the proof of this result in our simple case of Bernoulli uniform measures. In Section \ref{s3.2} we prove the second-order Boltzmann-Gibbs principle, only relying in Propositions \ref{p7} and \ref{p9}. The proof follows from a multiscale analysis introduced in \cite{Gon}. The multiscale analysis has two steps. The first one, which we call the {\em seed} is reminiscent of the well-known one-block estimate. In the second step we apply a key {\em iterative bound} to go from a microscopically big block to a macroscopically small block. 

In Section \ref{s2} we prove Theorem \ref{t1}. The proof follows the classical scheme to prove convergence theorems in probability. In Section \ref{s2.1} we prove tightness of the density fluctuation fields and in Section \ref{s2.2} we prove that any limit point is a stationary energy solution of the KPZ equation. We conjecture that energy solutionst starting from the stationary state are unique in distribution. Conditioned to this uniqueness result, convergence follows. 

In Section \ref{s4.0} we prove the convergence results for the current and height fluctuation fields. The central limit theorem for the current follows from an idea of Rost and Vares \cite{RV} and we follow the approach of \cite{JL} and \cite{Gon}. The height fluctuation field formally corresponds to the integral of the density fluctuation field. The convergence of the height fluctuation field does not follow directly from the convergence of the density, since we need to deal with the constant of integration, which is a non-trivial process that we relate with a sort of mollified current process.

\section{Notation and results}
\label{s1}

\subsection{The model}
\label{s1.1}
Let $\Omega =\{0,1\}^{\bb Z}$ be the state space of a continuous-time Markov chain $\eta_t$ which we will define as follows. We say that a function $f: \Omega \to \bb R$ is {\em local} if there exists $R= R(f)>0$ such that $f(\eta)=f(\xi)$ for any $\eta, \xi \in \Omega$ such that $\eta(x) = \xi(x)$ whenever $|x| \geq R$. In other words, we say that $f$ is local if $f(\eta)$ depends only on a finite number of coordinates of $\eta$ (in this case at most $2R+1$). Let $c: \Omega \to \bb R$ be a non-negative function. We assume the following conditions on $c$:

\begin{itemize}
\item[i)] {\em Ellipticity}: There exists $\epsilon_0>0$ such that $\epsilon_0 \leq c(\eta) \leq \epsilon_0^{-1}$ for any $\eta \in \Omega$. 
\item[ii)] {\em Finite range}: The function $c(\cdot)$ is local.

\item[iii)] {\em Reversibility}: For any $\eta, \xi \in \Omega$ such that $\eta(x) = \xi(x)$ whenever $x \neq 0,1$, $c(\eta) = c(\xi)$. 

\end{itemize}

For any $x \in \bb Z$ let $\tau_x : \Omega \to \Omega$ be the translation in $x$: $\tau_x\eta(z) = \eta(x+z)$ for any $\eta \in \Omega$ and any $z \in \bb Z$. For a function $f:\Omega \to \bb R$ we define $\tau_x f:\Omega \to \bb R$ as  $\tau_x f(\eta) = f(\tau_x \eta)$ for any $\eta \in \Omega$. We will also assume a fourth condition, which is the most restrictive one:

\begin{itemize}
\item[iv)] {\em Gradient condition}: There exists a local function $h: \Omega \to \Omega$ such that $c(\eta)(\eta(1)-\eta(0)) = \tau_1 h(\eta)- h(\eta)$ for any $\eta \in \Omega$. 
\end{itemize}

The exclusion process (possibly asymmetric) with speed change is defined as the Markov process $\{\eta_t^n;t \geq 0\}$ generated by the operator $L_n$, whose action over local functions $f:\Omega \to \bb R$ is given by 
\[
L_n f(\eta) = n^2 \sum_{x \in \bb Z} c_x(\eta) \big\{p_n \eta(x)(1-\eta(x+1)) + q_n \eta(x+1)(1-\eta(x))\big\} \nabla_{x,x+1} f(\eta),
\]
where $n \in \bb N$ 
\footnote{Here and below we use the convention $\bb N =\{1,2,\dots\}$},
$c_x(\eta) = \tau_x c(\eta)$, $\nabla_{x,x+1} f(\eta) = f(\eta^{x,x+1})-f(\eta)$, $p_n$ and $q_n$ are non-negative constants such that $p_n+q_n=1$ (eventually we will choose $p_n-q_n =a/\sqrt n$ with $a \neq 0$) and $\eta^{x,x+1}$ is given by
\[
\eta^{x,x+1}(z) =
\begin{cases}
\eta(x+1),& z=x\\
\eta(x), & z=x+1\\
\eta(z), & z\neq x,x+1.\\
\end{cases}
\]

In order to explain the meaning of conditions i)-iv) let us assume by now that $p_n=q_n=1/2$. In this case the process $\eta_t^n$ is said to be {\em symmetric}. Condition i) ensures that the process is well defined (see Chapter 1 of \cite{Lig} for a comprehensive discussion) for any choice of $p_n$, $q_n$.
For $\rho \in [0,1]$ let $\nu_\rho$ be the Bernoulli product measure in $\Omega$ of parameter $\rho$. This means that for any two finite, disjoint sets $A, B \in \bb Z$,
\[
\nu_\rho(\eta(x)=1 \;\forall x \in A, \eta(y) =0 \;\forall y \in B) = \rho^{|A|}(1-\rho)^{|B|}, 
\]
where $|A|$, $|B|$ denote the cardinality of the sets $A$, $B$ respectively. Under condition iii), the measures $\{\nu_\rho; \rho \in [0,1]\}$ are invariant and reversible with respect to the evolution of $\eta_t^n$. Under condition i), these measures are also ergodic with respect to the evolution of $\eta_t^n$. Condition i) also tells us that the dynamics of $\eta_t^n$ is comparable to the dynamics of the simple exclusion process without speed change: let $L_n^{ex}$ be the generator associated to the choice $c(\cdot) \equiv 1$. Then,
\begin{equation}
\label{ec0.1}
\epsilon_0 \int f(-L_n f)d\nu_\rho \leq \int f(-L_n^{ex}f)d\nu_\rho \leq \epsilon_0^{-1} \int f(-L_n f)d\nu_\rho
\end{equation}
for any local function $f$. We will return to the meaning of this bound later. 

When the process $\eta_t^n$ is asymmetric (that is, when $p_n \neq q_n$), it is not true in general that the measures $\nu_\rho$ are invariant with respect to the evolution of $\eta_t^n$. In fact, according to \cite{Nag}, the family $\{\nu_\rho; \rho \in [0,1]\}$ is invariant with respect to the evolution of $\eta_t^n$ if and only if the condition iv) is satisfied. It is exactly for this reason that we assume the restrictive condition iv). In this case, the measures $\nu_\rho$ are no longer reversible with respect to the dynamics.

\subsection{Scaling limits}
\label{s1.2}
In this section we recall various scaling limits previously obtained for the density of particles with respect to the process $\eta_t^n$. These results are known in the literature as {\em hydrodynamic limits}. First we recall a law of large numbers for the density of particles. Let $\pi_t^n(dx)$ be the positive measure in $\bb R$ defined as
\[
\pi_t^n(dx) = \frac{1}{n} \sum_{x \in \bb Z} \eta_t^n(x) \delta_{x/n}(dx),
\]
where $\delta_{x/n}(dx)$ is the Dirac mass at $x/n \in \bb R$. The process $\pi_t^n(dx)$ is known as the empirical density measure associate to the process $\eta_t^n$.  When the distribution of $\eta_0^n$ is equal to $\nu_\rho$ with $\rho \in [0,1]$ then the distribution of $\eta_t^n$ is also equal to $\nu_\rho$ for any later time $t>0$. In particular, the sequence of measures $\{\pi_t^n(dx); n\in \bb N\}$ converges in probability to the measure $\rho dx$ when $n \to \infty$ for any $t \geq 0$. 

Let $u_0: \bb R \to [0,1]$ be a strictly positive, continuous by parts function such that there exists $\rho \in (0,1)$ for which $\int |u_0(x) -\rho|dx<+\infty$. Let $\{\mu^n;n \in \bb N\}$ be a sequence of Bernoulli product measures in $\Omega$, defined by the relation $\mu^n(\eta(x)=1) = u_0(x/n)$.
Let us denote by $\bb P_{\mu^n}$ the distribution of the process $\eta_t^n$ with initial distribution $\mu^n$. The following theorem is known as the hydrodynamic limit of the process $\eta_t^n$:

\begin{proposition}[\cite{GPV,KOV}]
\label{p1}
Let us take $p_n-q_n=a/n$ with $a \in \bb R$ and $n$ big enough. Let $u_0: \bb R \to [0,1]$ and $\{\mu^n;n \in \bb N\}$ be as above. Then $\pi_t^n(dx)$ converges in probability to the measure $u(t,x) dx$ with respect to $\bb P_{\mu^n}$, where $\{u(t,x); t \geq 0, x \in \bb R\}$ is the solution of the  {hydrodynamic equation}
\begin{equation}
\label{echid}
\begin{cases}
\partial_t u &= 1/2 \Delta \varphi(u) - a \nabla \beta(u)\\
u(0,\cdot) &= u_0(x)
\end{cases}
\end{equation}
and $\varphi(\rho) = \int h d\nu_\rho$, $\beta(\rho) = \chi(\rho)\int c d\nu_\rho$.
\end{proposition}

The Einstein relation, which holds in great generality, states that $\beta(\rho) = \chi(\rho) \varphi'(\rho)$. The quantity $\chi(\rho) = \chi(\rho)$ is known in the literature as the {\em conductivity} or {\em susceptibility} of the system. As we said before, the gradient condition iv) is very restrictive. From the point of view of modeling, this condition is not so restrictive. In fact, for any $m \in \bb N$ there exists a choice for $c(\eta)$ such that $\varphi(\rho) = \rho^m$ (see \cite{GLT}). By linearity, for any polynomial $q(\rho)$ there are a constant $K$ and rate $c(\eta)$ such that $\varphi(\rho)=q(\rho)+K$.

The symmetric case $a=0$ in Proposition \ref{p1} was treated in \cite{GPV} for a model with interactions of Ginzburg-Landau type, and the weakly asymmetric case was treated in \cite{KOV}. In both cases, the gradient condition iv) is fundamental. The non-gradient method developed in \cite{Qua,Var} allows to generalize this theorem (in the case $a=0$) to the case on which iv) is not satisfied. In that case the function $\varphi(\rho)$ is given in terms of a variational formula, and in particular $\varphi(\rho)$ is not explicit. Under the additional hypothesis $\varphi(\cdot) \in \mc C^1((0,1))$, \cite{FUY} proved Proposition~\ref{p1} when $a=0$ for the very same model considered here. Later, in \cite{Ber} it is proved that $\varphi(\cdot) \in \mc C^\infty((0,1))$, closing the proof of Proposition~\ref{p1} for this model.

In this article we are interested in a central limit theorem for the density of particles, with respect to the hydrodynamic limit considered above. The current state of the art restrict ourselves to the equilibrium situation, that is, when the initial distribution of $\eta_t^n$ is equal to $\nu_\rho$ for some $\rho \in (0,1)$ (the cases $\rho =0,1$ being trivial). Let us fix now and for the rest of the paper a density $\rho \in (0,1)$ and let $\bb P_n$ be the distribution of $\eta_t^n$ with initial condition $\nu_\rho$. We denote by $\bb E_n$ the expectation with respect to $\bb P_n$. Let $\mc S(\bb R)$ be the Schwartz space of test functions and let $\mc S'(\bb R)$ be the space of tempered distributions in $\bb R$, which corresponds to the topological dual of $\mc S(\bb R)$ (we give more precise definitions in Sect.~\ref{s1.3}). The fluctuation field $\{ \mc Y_t^{n,0}; t \geq 0\}$ is defined as the $\mc S'(\bb R)$-valued process given by
\[
\mc Y_t^{n,0}(G) = \frac{1}{\sqrt n} \sum_{x \in \bb Z} \big(\eta_t^n(x) -\rho \big) G(x/n)
\]
for any function $G \in \mc S(\bb R)$.
Calculating  the characteristic function of the random variable $\mc Y_t^{n,0}(G)$, it is easy to see that for any fixed time $t \geq 0$, the process  $\mc Y_t^{n,0}$ converges in distribution to a spatial white noise of variance $\chi(\rho)$. For the sequence of {\em processes} $\{\mc Y_t^{n,0};n \in \bb N\}$ the scaling limit is the following:

\begin{proposition}[\cite{D-MPSW,Cha, Fun}]
\label{p2}
Let us take $p_n-q_n=a/n$ with $a \in \bb R$ and $n$ big enough.
The sequence $\{\mc Y_t^{n,0};n \in \bb N\}$ converges in distribution with respect to the $J_1$-Skorohod topology in $\mc D([0,\infty), \mc S'(\bb R))$ to the process $\mc Y_t^0$, solution of the Ornstein-Uhlenbeck equation
\begin{equation}
\label{ec1}
d \mc Y_t^0 = \frac{1}{2} \varphi'(\rho) \Delta \mc Y_t^0 dt - a\beta'(\rho) \nabla \mc Y_t^0 dt
	+\sqrt{\chi(\rho)\varphi'(\rho)} \nabla d \mc W_t,
\end{equation}
where $\mc W_t$ is a space-time white noise of unit variance.
\end{proposition}

The proof of this result relies on a replacement known as the {\em Botzmann-Gibbs principle}, which  was introduced by H. Rost \cite{BR}. When the gradient condition iv) is satisfied, this theorem was proved by \cite{D-MPSW} for the case $a=0$. Later a more robust proof was given by \cite{Cha}, which can be extended to the case $a \neq 0$. When the gradient condition iv) is not satisfied, a sketch of the proof is given in \cite{Fun}, following the method outlined in \cite{Cha2}.

The drift term in \eqref{ec1} is irrelevant in the following sense. Performing a Galilean coordinate transformation, the drift term in equation \eqref{ec1} can be removed. This is clear if $\beta'(\rho)$ is equal to 0, in which case the equation \eqref{ec1} does not depend on $a$. In other words, the drift term does not give rise to stochastic fluctuations of the density field. In a more precise way, let us define the modified fluctuation field $\mc Y_t^n$ as
\[
\mc Y_t^n(G) = \frac{1}{\sqrt n} \sum_{x \in \bb Z} \big(\eta_t^n(x) -\rho\big) G(x/n - v(\rho)t),  
\] 
where $v(\rho) = a \beta'(\rho)$ is the velocity associated to the system. Starting from Proposition~\ref{p2} we see that the sequence $\{\mc Y_t^n; n\in \bb N\}$ converges to the process $\mc Y_t$, solution of the equation
\begin{equation}
\label{ec2}
d\mc Y_t = \frac{1}{2} \varphi'(\rho) \Delta \mc Y_t dt + \sqrt{\chi(\rho)\varphi'(\rho)} \nabla d \mc W_t,
\end{equation}
which corresponds to equation \eqref{ec1} with $a=0$. 

As we mentioned before, the asymmetry is too weak in order to induce a stochastic fluctuation in the density of particles. According to \cite{BG}, a non trivial fluctuation due to the asymmetry appears when $p_n -q_n = a/\sqrt n$, on which case the limiting process  $\mc Y_t$ has a qualitatively different evolution from the Ornstein-Uhlenbeck process  (solution of \eqref{ec2}). For this choice of the asymmetry, the modified fluctuation field is given by
\[
\mc Y_t^n(G) = \frac{1}{\sqrt n} \sum_{x \in \bb Z} \big(\eta_t^n(x) -\rho\big) G(x/n - v(\rho)tn^{1/2}).  
\]
\begin{proposition}[\cite{BG}]
\label{p3}
Let us take $p_n-q_n=a/\sqrt n$ and $c \equiv 1$. Let $u_0: \bb R \to [0,1]$ be a continuous function and let $\mu^n$ be the measure appearing in Proposition~\ref{p1}. Then the process $\mc Y_t^n$ converges in distribution with respect to $\bb P_{\mu^n}$ to the Cole-Hopf solution of the KPZ equation
\begin{equation}
\label{ec3}
d \mc Y_t = \frac{1}{2} \Delta \mc Y_t  dt + a \nabla \mc Y_t^2 dt +\sqrt{u_t(1-u_t)} \nabla d \mc W_t
\end{equation}
with initial condition $\mc W^{u_0}$, where $u_t$ is the solution of the hydrodynamic equation \eqref{echid} and $\mc W^{u_0}$ is a Gaussian process in $\bb R$ of mean zero and covariance given by $u_0(x)(1-u_0(x))\delta(x,y)$. 
\end{proposition}

From now on and up to the end of the paper, we take $p_n-q_n = a/\sqrt n$ and for ease of notation we assume that $v(\rho)=0$.
In Section \ref{s1.3} we will explain better what do we call a Cole-Hopf solution of equation \eqref{ec3}. The KPZ equation was introduced in \cite{KPZ} as a continuum model for surface growth. The proof of this theorem makes explicit use of the fact that the simple exclusion process is a totally integrable system. In a more precise way, in \cite{BG} the authors exploit a non-linear transformation of the process, discovered by Gartner \cite{Gar} which linearizes the evolution of $\mc Y_t^n$. This transformation is a microscopic analogous of the Cole-Hopf transformation, and it is of help only in the case $c \equiv 1$.

The main result of this article is a generalization of Proposition~\ref{p3} in the equilibrium case. But in order to write the result in a precise way, we need to introduce various definitions.

\subsection{KPZ equation}

\label{s1.3}

The KPZ equation, which was introduced by Kardar, Parisi and Zhang in the celebrated paper \cite{KPZ} formally reads
\begin{equation}
\label{ec.KPZ0}
d h_t = \frac{ \varphi'(\rho)}{2} \Delta h_t dt - \frac{a \beta''(\rho)}{2} \big(\nabla h_t\big)^2 dt + \sqrt{\chi(\rho)\varphi'(\rho)} d \mc W_t,
\end{equation}
where $h_t$ is a stochastic process with values on the set $\mc C(\bb R)$ of continuous functions in $\bb R$ and $\mc W_t$ is a space-time white noise, that is, a Gaussian process of mean zero and covariance $\delta(x-x') \delta(t-t')$. Formally defining $\mc Y_t = \nabla h_t$ we obtain the conservative KPZ equation:
\begin{equation}
\label{ec.KPZ}
dÊ\mc Y_t = \frac{\varphi'(\rho)}{2} \Delta \mc Y_t dt - \frac{a \beta''(\rho)}{2} \nabla \mc Y_t^2 dt + \sqrt{\chi(\rho)\varphi'(\rho)} \nabla d \mc W_t.
\end{equation}
In the literature this equation is sometimes called stochastic Burgers equation. The stochastic Burgers equation has been studied in detail for the case on which the noise has the strength enough to regularize the solutions of the equation. Therefore, we prefer to reserve the term stochastic Burgers equations for the well studied cases on which the noise regularizes the equation. 
It turns out that, at least formally, the spatial white noise of variance $\chi(\rho)$ is invariant under the evolution of \eqref{ec.KPZ}. 
Let $\mc C_c^\infty(\bb R)$ be the space of infinitely differentiable functions $f: \bb R \to \bb R$ of compact support. For each $l,m \in \bb N$ we define in $\mc C^\infty_c(\bb R)$ the norms
\[
\|f\|_{(l,m)} = \sup_{x \in \bb R} \big|x^l f^{(m)}(x)\big|,
\]
where $f^{(m)}$ is the $m$-th derivative of $f$. Now we define the distance $d_{\mc S}: \mc C^\infty_c(\bb R) \times \mc C^\infty_c(\bb R) \to \bb R$ as
\[
d_{\mc S}(f,g) = \sum_{l,m \in \bb N} \frac{1}{2^{l+m}} \min\{\|f-g\|_{(l,m)},1\}.
\]
for any $f, g \in \mc C_c^\infty(\bb R)$. The Schwartz space $\mc S(\bb R)$ is defined as the closure of $\mc C^\infty_c(\bb R)$ with respect to the distance $d_{\mc S}(\cdot,\cdot)$. The space $\mc S(\bb R)$ coincides with the set of infinitely differentiable functions $f: \bb R \to \bb R$ such that $\|f\|_{(l,m)}<+\infty$ for any $l,m \in \bb N$. The space of {\em tempered distributions} $\mc S'(\bb R)$ is defined as the topological dual of $\mc S(\bb R)$. For $T \in (0,\infty)$, let $\mc C([0,T], \mc S'(\bb R))$ be the space of continuous functions from $[0,T]$ to $\mc S'(\bb R)$. 

The space $\mc C([0,T],\mc S'(\bb R))$ is the space on which the solutions of the KPZ equation \eqref{ec.KPZ} will live. For $\epsilon >0$ we define $i_\epsilon(x):\bb R \to \bb R$ by 
\[
i_\epsilon(x)(y) = \epsilon^{-1}\mathbf{1}(x < y \leq x+\epsilon).
\]
We say that a process $\{\mc Y_t; t\in [0,T]\}$ with trajectories in $\mc C([0,T],\mc S'(\bb R))$ and adapted to some natural filtration $\{\mc F_t; t \in [0,T]\}$ is a {\em weak solution} of the KPZ equation \eqref{ec.KPZ} if:
\begin{itemize}
\item[i)]
There exists a process $\{\mc A_t; t \in [0,T]\}$ with trajectories in $\mc C([0,T], \mc S'(\bb R))$ and adapted to $\{\mc F_t; t\in [0,T]\}$ such that for any $G \in \mc S(\bb R)$,
\begin{equation}
\label{ec16}
\lim_{\epsilon \to 0} \int_0^t \int_{\bb R} \mc Y_s(i_\epsilon(x))^2 \frac{G(x+\epsilon)-G(x)}{\epsilon} dx ds = \mc A_t(G).
\end{equation}
\item[ii)]
For any function $G \in \mc S(\bb R)$ the process
\begin{equation}
\label{ec17}
M_t(G) = \mc Y_t(G) - \mc Y_0(G) -\frac{\varphi'(\rho)}{2} \int_0^t \mc Y_s(G'') ds - \frac{a \beta''(\rho)}{2} \mc A_t(G)
\end{equation}
is a martingale of quadratic variation $\chi(\rho) \varphi'(\rho) t \|G\|_1^2$.
\end{itemize}

Here and below we write $\|G\|_1^2 = \int G'(x)^2dx$. This notion of solution of the KPZ equation is new. Moreover, it seems to us that any effort to define a meaningful notion of solution of \eqref{ec.KPZ} has not been well succeeded. Now we will introduce a stronger notion of solution, which captures well some of the particularities of the solutions of \eqref{ec.KPZ}. Let $\{\mc Y_t; t \in [0,T]\}$ be a weak solution of \eqref{ec.KPZ}. For $0\leq s<t\leq T$, let us define the fields
\[
\mc I_{s,t}(G) = \int_s^t \mc Y_u(G'')du,
\]
\[
\mc A_{s,t}(G) = \mc A_t(G) - \mc A_s(G),
\]
\[
\mc A_{s,t}^\epsilon(G) = \int_s^t \int_{\bb R} \mc Y_u(i_\epsilon(x))^2 \frac{G(x+\epsilon)-G(x)}{\epsilon} dx du.
\]

We say that $\{\mc Y_t;  \in [0,T]\}$ is an {\em energy solution} of the KPZ equation \eqref{ec.KPZ} if there exists a constant $\kappa >0$ such that
\[
E[\mc I_{s,t}(G)^2] \leq \kappa (t-s) \|G\|_1^2
\]
and 
\[
E[(\mc A_{s,t}(G) - \mc A_{s,t}^\epsilon(G))^2] \leq \kappa \epsilon (t-s) \|G\|_1^2
\]
for any $0\leq s<t\leq T$, any $\epsilon \in (0,1)$ and any $G \in \mc S(\bb R)$. We say that a weak solution $\{\mc Y_t ; t \in [0,T]\}$ is a stationary solution if for any $t \in [0,T]$ the $\mc S'(\bb R)$-valued random variable $\mc Y_t$ is a white noise of variance $\chi(\rho)$. Now we are ready to state the main result of this article.
\begin{theorem}[Equilibrium fluctuations]
\label{t1}
The sequence of proceses $\{\{\mc Y_t^n; t \in [0,T]\}; n \in \bb N\}$ is tight in $\mc D([0,T],\mc S'(\bb R))$. Moreover, any limit point of $\mc Y_t^n$ is a stationary energy solution of \eqref{ec.KPZ}. 
\end{theorem}

An immediate consequence of this result is the existence of weak solutions of the KPZ equation. Let $\mc Y_t$ be a limit point of $\mc Y_t^n$. Since the measure $\nu_\rho$ is invariant under the evolution of $\eta_t^n$, for any fixed time $t \in [0,T]$ the $\mc S'(\bb R)$-valued random variable $\mc Y_t$ is  a white noise of variance $\chi(\rho)$. The following corollary states some properties of the sample paths of the process $\mc Y_t$.

\begin{corollary}
\label{cor1}
For any limit point $\{\mc Y_t; t \in [0,T]\}$ of $\{\{\mc Y_t^n; t \in [0,T]\}; n \in \bb N\}$, there is a finite constant $c>0$ such that the process $\{\mc A_t; t \in [0,T]\}$ defined as above satisfies the moment bound
\[
\bb E[\mc A_{s,t}(G)^2] \leq c  |t-s|^{3/2} \|G\|_1^2.
\]

Moreover, for any $\gamma \in (0,1/4)$ and any $G \in \mc S(\bb R)$ the real-valued process $\{\mc Y_t(G); t \in [0,T]\}$ is H\"older-continuous of order $\gamma$.
\end{corollary}

The only rigorous result about existence of solutions of \eqref{ec.KPZ} we know is the work of Bertini and Giacomin \cite{BG}. Let us describe this result in a precise way. Let $\mc C_+(\bb R)$ be the set of positive, continuous functions $f: \bb R \to \bb R$. We say that a process $\{\mc Z_t; t \in [0,T]\}$ with trajectories in $\mc C([0,T]; \mc C_+(\bb R))$ is a {\em mild solution} of the stochastic heat equation
\begin{equation}
\label{ec18}
d \mc Z_t = \frac{\varphi'(\rho)}{2} \Delta \mc Z_t dt + a \beta''(\rho) \sqrt{\frac{\chi(\rho)}{\varphi'(\rho)}} \mc Z_t d \mc W_t
\end{equation}
if  the process $\mc Z_t$ satisfies the integral equation
\[
\mc Z_t = K_t \ast \mc Z_0 - \int_0^t K_{t-s} \ast \mc Z_s d \mc W_s,
\]
where $ K_t(x) = (2\pi \varphi'(\rho)t)^{-1/2} \exp\{-x^2/2\varphi'(\rho)t\}$ is the heat kernel and $\ast$ denotes convolution. We say that a process $\{h_t; t \in [0,T]\}$ is a {\em Cole-Hopf solution} of \eqref{ec.KPZ0} if $h_t = -\varphi'(\rho)/a\beta''(\rho) \log \mc Z_t$ for any $t \in [0,T]$, where $\mc Z_t$ is a mild solution of the stochastic equation \eqref{ec18}. Defining $\mc Y_t = \nabla h_t$ in the distributional sense, we say that $\{\mc Y_t; t \in [0,T]\}$ is a Cole-Hopf solution of the KPZ equation \eqref{ec.KPZ} if $\{h_t; t \in [0,T]\}$ is a Cole-Hopf solution of \eqref{ec.KPZ0}.

An important issue raised in \cite{BG} (see the remark after Theorem~2.1 in that article) is wheter a Cole-Hopf solution of \eqref{ec.KPZ} actually satisfies this equation in any meaningful sense. Since Cole-Hopf solutions of \eqref{ec.KPZ} arise as scaling limits of density fields in the weakly asymmetric exclusion process, Theorem~\ref{t1} combined with the results in \cite{BG} has the following consequence.

\begin{theorem}
\label{t1.1}
The Cole-Hopf solution of \eqref{ec.KPZ} with initial distribution given by a spatial white noise of variance $\chi(\rho)$ is an energy solution of \eqref{ec.KPZ}. 
\end{theorem}

\subsection{Current fluctuations and growing interfaces}
\label{s1.4}
For each $x \in \bb Z$ and any $t \in [0,\infty)$, let $J_t^n(x)$ be the current of particles through sites $x$ and $x+1$ up to time $t$. That is, $J_t^n(x)$ counts the number of particles passing between sites $x$ and $x+1$. The local conservation of  the number of particles plus the fact that the jumps are only to nearest neighbors are equivalent to the continuity equation $\eta_t^n(x)-\eta_0^n(x) = J_t^n(x-1)-J_t^n(x)$. In other words, ``what we have now minus what we had before is equal to what came in minus what went out''. The height function $h^n_t=\{h^n_t(x); x \in \bb Z\}$ is defined by
\begin{equation}
h^n_t(x) =
\begin{cases}
\label{ec.height}
J_t^n(0) - \sum_{y=1}^x \eta_t^n(y), & x >0\\
J_t^n(0),& x=0\\
J_t^n(0) +  \sum_{y=-x+1}^{-1} \eta_t^n(y), & x<0.
\end{cases}
\end{equation}
The continuity equation shows the relation $h_t^n(x) = h_0^n(x) + J_t^n(x)$. Notice that the process $\eta_t^n$ corresponds to the increments of the process $h_t^n$. In fact, $\eta_t^n(x) = h_t^n(x-1)-h_t^n(x)$. We write this relation in the symbolic form $\eta_t^n = -\nabla h^n_t$. We can interpret $h^n_t$ as the evolution of a particle system on which particles are created or annihilated, but on which particles never move. A jump of a particle from site $x$ to site $x+1$ in the exclusion process corresponds to  the creation of a particle at site $x$, and a jump from $x+1$ to $x$ corresponds to the  annihilation of a particle at site $x$. Therefore, in order to have a creation at site $x$, it is necessary to have $h_t^n(x-1)-h_t^n(x)=1$, $h_t^n(x)-h_t^n(x+1)=0$. And in order to have an annihilation at site $x$, it is necessary to have $h_t^n(x-1)-h_t^n(x)=0$, $h_t^n(x)-h_t^n(x+1)=1$. The creation rate, given this necessary condition, is equal to $n^2 p_n \gamma_x(h^n_t)$, where $\gamma_x(h^n_t) = c_x(-\nabla h^n_t)$. And the annihilation rate is equal to $n^2 q_n \gamma_x(h^n_t)$. If $a>0$, we say that $h^n_t$ is a {\em growing interface model}. It was in the context of growing interfaces that the KPZ equation \eqref{ec.KPZ0} (and also \eqref{ec.KPZ}) were introduced in the original work of Kardar, Parisi and Zhang \cite{KPZ}. 

Let us define the rescaled interface field as the process $\theta_t^n = \{\theta_t^n(x); x \in \bb R\}$ by taking $\theta_t^n(x/n) = n^{-1/2}\big(h_t^n(x)-\bb E_n[h_t^n(x)]\big)$ for $x \in \bb Z$ and $\theta_t^n(x) = \theta_t^n(\lfloor nx\rfloor/n)$ for $x \notin n^{-1} \bb Z$. We could also extend $\theta_t^n$ to $\bb R$ by linear interpolation, but in our setting it will prove to be more convenient to define $\theta_t^n$ as we do here. Notice that
\[
\theta_t^n(x) =  n^{-1/2} \big(J_t^n(0) - \bb E_n[J_t^n(0)]\big) - \mc Y_t^n(\mathbf{1}_{(0,x]})
\]
and in particular $\theta_t^n(0) =  n^{-1/2}(J_t^n(0) - \bb E_n[J_t^n(0)])$.
Let $H_0$ be the Heaviside function, that is, $H_0(x) = \mathbf{1}_{(0,\infty)}(x)$.
Our first result corresponds to a limit theorem for the current $J_t^n(0)$. 
\begin{theorem}
\label{t4}Let $\mc Y_t$ be a limit point of $\mc Y_t^n$, and denote by $n'$ a subsequence such that $\mc Y_t^{n'} \to \mc Y_t$. Then the process $\mc J_t(0) = \mc Y_t(H_0)-\mc Y_0(H_0)$ is well defined and 
\[
\lim_{n' \to \infty}\theta_t^{n'}(0) = \mc J_t(0)
\]
in the sense of convergence of finite-dimensional distributions.
\end{theorem}

The notion of energy solutions of the equation \eqref{ec.KPZ} defined in Section \ref{s1.3} can be extended in a natural way to the equation \eqref{ec.KPZ0}. We say that a process $\{\theta_t; t \in [0,T]\}$ with trajectories in $\mc C([0,T], \mc S'(\bb R))$ and adapted to some standard filtration $\{\mc F_t; t \in [0,T]\}$ is a weak solution of \eqref{ec.KPZ0} if
\begin{itemize}
\item[i)]
There exists a process $\{\mc B_t; t \in [0,T]\}$ with trajectories in $\mc C([0,T], \mc S'(\bb R))$ and adapted to $\{\mc F_t; t\in [0,T]\}$ such that for any $G \in \mc S(\bb R)$,
\[
\lim_{\epsilon \to 0} \int_0^t \int_{\bb R}  \int_x^{x+\epsilon} \frac{G(y)}{\epsilon} \Big\{\Big(\frac{\theta_s(x+\epsilon)-\theta_s(x)}{\epsilon}\Big)^2-\frac{\chi(\rho)}{\epsilon}\Big\}  dy dx ds = \mc B_t(G).
\]
\item[ii)]
For any function $G \in \mc S(\bb R)$ the process
\[
\mc M_t(G) = \<\theta_t,G\> -\<\theta_0,G\> - \frac{\varphi'(\rho)}{2} \int_0^t \<\theta_s,G''\> ds - \frac{a \beta''(\rho)}{2} \mc B_t(G)
\]
is a martingale of quadratic variation $\chi(\rho) \varphi'(\rho) t \int G(x)^2 dx$.
\end{itemize}

Let $\{\theta_t;t \in [0,T]\}$ be a weak solution of \eqref{ec.KPZ0}. For $0\leq s<t\leq T$, let us define the fields
\[
\mc B_{s,t}(G) = \mc B_t(G) - \mc B_s(G),
\]
\[
\mc B_{s,t}^\epsilon(G) = \int_s^t \int_{\bb R} \int_x^{x+\epsilon} \Big\{\Big(\frac{\theta_t(x+\epsilon)-\theta(x)}{\epsilon}\Big)^2-\frac{\chi(\rho)}{\epsilon}\Big\}G(x) dy dx du.
\]

We say that $\{\theta_t;  \in [0,T]\}$ is an {\em energy solution} of the KPZ equation \eqref{ec.KPZ0} if there exists a constant $\kappa >0$ such that
\[
E\Big[\Big(\int_s^t \<\theta_u,G''\>du\Big)^2\Big] \leq \kappa (t-s) \int G(x)^2 dx
\]
and 
\[
E[(\mc B_{s,t}(G) - \mc B_{s,t}^\epsilon(G))^2] \leq \kappa \epsilon (t-s) \int G(x)^2 dx
\]
for any $0\leq s<t\leq T$, any $\epsilon \in (0,1)$ and any $G \in \mc S(\bb R)$.
\begin{theorem}
\label{t6}
The sequence of processes $\{\{\theta_t^n;t \in [0,T]\}; n \in \bb N\}$ is tight with respect to the $J$-Skorohod topology of $\mc D([0,T]; \mc S'(\bb R))$. Moreover, any limit point $\{\theta_t; t \in [0,T]\}$ of this sequence is an energy solution of the KPZ equation \eqref{ec.KPZ0} with initial distribution given by a two-sided Brownian motion of variance $\chi(\rho)$.
\end{theorem}

\section{Second-order Boltzmann-Gibbs principle}
\label{s3}

Let $f:\Omega \to \bb R$ be a local function and let us define $\psi(\rho) = \int fd\nu_\rho$. In \cite{D-MPSW}, the following theorem is proved.

\begin{proposition}
\label{p4}
Let us assume that $a=0$. Then, for any function $H \in \mc S(\bb R)$ we have
\[
\lim_{n \to \infty} \bb E_n\Big[\Big(\int_0^t \frac{1}{\sqrt n} \sum_{x \in \bb Z} \big(\tau_x f(\eta_s^n) -\psi(\rho) - \psi'(\rho)(\eta_s^n(x)-\rho)\big)H(x/n)ds\Big)^2\Big] =0.
\]
\end{proposition}

This result is the celebrated Boltzmann-Gibbs principle introduced by Rost. It roughly says that the fluctuation field associated to $f$ is asymptotically equivalent to a multiple of the density fluctuation field $\mc Y_t^n$. The idea is the following. Particles are neither created nor destroyed by the dynamics. Therefore, in order to equilibrate a local fluctuation of the number of particles, it is necessary to transport it to another region. The density of particles is the only locally conserved quantity of the system. Due to the ellipticity condition i), the process has good ergodic properties. Therefore, a fluctuation of a non conserved quantity will be locally equilibrated. If we look at the process in the right time scaling, the only fluctuation we will see will be the fluctuation of the density; other fluctuations being too fast to be observed in that scale. 

Notice that if $\psi'(\rho)=0$, Proposition~\ref{p4} does not give a lot of information: it simply asserts that the fluctuation field associated to $f$ asymptotically vanishes. Let $f: \Omega \to \bb R$ be a local function such that $\psi'(\rho)=0$. For ease of notation we also assume $\psi(\rho)=0$. For $G \in \mc S(\bb R)$ and $n \in \bb N$ we define the function $\nabla_\cdot^n G: \bb Z \to \bb R$ by
\[
\nabla_x^n G = n\Big\{G\Big(\frac{x+1}{n}\Big) - G\Big(\frac{x}{n}\Big)\Big\}.
\]
In other words, $\nabla_x^n G$ is a discrete approximation of $G'(x/n)$.
Let us define the field $\mc A_t^n$ with values in $\mc S'(\bb R)$ as
\[
\mc A_t^n(G) = \int_0^t \sum_{x \in \bb Z} \tau_x f(\eta_s^n) \nabla_x^n G ds
\]
Notice that the field $\mc A_t^n$ is analogous to the integral field appearing in Proposition~\ref{p4} multiplied by $\sqrt n$ and evaluated in $G'$. It turns out that the prefactor $\sqrt n$ will make appear  a non-vanishing limit for $\mc A_t^n$. Another important observation is that, unlike the case in Proposition~\ref{p4}, the field inside the integral {\em does not converge} to any limit as $n \to \infty$. In fact, its variance grows like $n$. Therefore, the convergence of $\mc A_t^n$ to a well defined limit will be a purely dynamical feature of the process $\eta_t^n$.

In order to state what we call the second-order Boltzmann-Gibbs principle, we need some notation. For $\epsilon >0$ we denote by $\epsilon \bb Z$ the set $\{\epsilon z; z \in \bb Z\}$. From now on, quantities like $\epsilon n$ are treated as if they were integers, meaning sometimes $\lfloor \epsilon n\rfloor$ and sometimes $\lceil \epsilon n \rceil$.

\begin{theorem}[Second-order Botzmann-Gibbs principle]
\label{t2}
For any function $G \in \mc S(\bb R)$,
\[
\lim_{\epsilon \to 0} \limsup_{n \to \infty} \bb E_n\Big[\Big(\mc A_t^n(G) - \frac{\psi''(\rho)}{2}\int_0^t \sum_{x \in \epsilon \bb Z} \mc Y_s^n(i_\epsilon(x))^2\big(G(x+\epsilon)-G(x)\big) ds\Big)^2\Big] =0.
\]
\end{theorem}

 This result is telling us {\em ``grosso modo''} that the field $\mc A_t^n$ is asymptotically equivalent to the field $(\mc Y_t^n)^2$. Since $\mc Y_t^n$ is a distribution, the square $(\mc Y_t^n)^2$ should be defined through some type of regularization, which is exactly what Theorem \ref{t2} is saying. According to Assing \cite{Ass1}, the expression $(\mc Y_t^n)^2$ can not be defined as the limit in some sense of the regularizing sequence $\mc Y_t^n(i_\epsilon(x))^2$, at least not as a distribution. 

The proof of Theorem~\ref{t2} is based on the multiscale analysis introduced in the article \cite{Gon}. We will divide the proof  into two parts. The basics elements used in the proof are the Kipnis-Varadhan inequality, a sharp estimate of the spectral gap of the generator $L_n$ restricted to finite boxes and the equivalence between the grandcanonical and the canonical ensembles. In Section \ref{s3.1} we recall the basic elements of the proof and in Section \ref{s3.2} we explain the multiscale analysis.

\subsection{Elements of proof}
\label{s3.1}
In this section we recall Kipnis-Varadhan and spectral gap inequalities, which will allow us to estimate the variance of various additive functionals related to the process $\mc A_t^n$. Our aim is to establish Proposition~\ref{p7} and Proposition~\ref{p9}. In Section \ref{s3.2} we will see that Theorem \ref{t2} follows from Proposition~\ref{p7} and \ref{p9} without further assumptions. Let $f:\Omega \to \bb R$ be a function in $L^2(\nu_\rho)$ such that $\int f d\nu_\rho =0$ for any $\rho \in [0,1]$. We start recalling Kipnis-Varadhan inequality.

\begin{proposition}[\cite{KV,CLO}]
\label{p5}
For any $T>0$,
\[
\bb E_n\Big[\sup_{0\leq t\leq T} \Big(\int_0^t f(\eta_s^n)ds\Big)^2\Big] \leq 20 T \|f\|_{-1,n}^2, 
\]
where
\[
\|f\|_{-1,n}^2 = \sup_{g \in L^2(\nu_\rho)} \big\{ 2\<f,g\>_\rho - \<g, -L_n g\>_\rho\big\}
\]
and $\<\cdot,\cdot\>_\rho$ denotes the inner product in $L^2(\nu_\rho)$.
\end{proposition}

This inequality was proved in the reversible case ($a=0$) by Kipnis and Varadhan \cite{KV} and in the general case by Chang, Landim and Olla \cite{CLO}.  Kipnis and Varadhan also proved that this bound is sharp in the reversible case.

This proposition is not very useful unless we have an effective way to compute the Sobolev norm $\|f\|_{-1,n}$. Let us recall the equivalence \eqref{ec0.1}. This equivalence leads us to define the {\em Dirichlet form} $\mc D: L^2(\nu_\rho) \to \bb R$ as
\[
\mc D(f) = \sum_{x \in \bb Z} \int \big(\nabla_{x,x+1} f(\eta)\big)^2 \nu_\rho(d\eta).
\]
Relation \eqref{ec0.1} now reads $\epsilon_0 n^2 \mc D(f) \leq \<f,-L_n f\>_\rho \leq \epsilon_0 n^2 \mc D(f)$ for any local function $f: \Omega \to \bb R$. For each $x \in \bb Z$, let us define
\[
\mc D_x(f) = \int \big(\nabla_{x,x+1} f(\eta)\big)^2 \nu_\rho(d\eta),
\]
so that $\mc D(f) = \sum_x \mc D_x(f)$.

Take $A \subseteq \bb Z$ and $f:\Omega \to \bb R$. We say that $\supp(f) \subseteq A$ if $f(\eta) = f(\xi)$ whenever $\eta(x) = \xi(x)$ for any $xÊ\in A$. Let us define $\mc F_A = \sigma(\eta(x); x \in A)$, the $\sigma$-algebra generated by the coordinates of $\eta$ in $A$. Notice that $\supp(f) \subseteq A$ if and only if $f$ is $\mc F_A$-mesurable. The function $\mc D_x(f)$ has the following convexity property: for $A \subseteq \bb Z$ and $x \in \bb Z$ such that $\{x,x+1\}\subseteq A$ we have
\[
\mc D_x(E[f|\mc F_A]) \leq \mc D_x(f).
\]
Here and below, conditional expectations will always be taken with respect to the measure $\nu_\rho$. 
The next estimate, known as the spectral gap inequality connects the variance of a local function $f$ with its Dirichlet form $\mc D(f)$:

\begin{proposition}[\cite{Qua,DS-C}]
\label{p6}
There exists a universal constant $\lambda_0$ such that for any $k \in \bb N$ and any function $f$ with $\supp(f) \subseteq \{1,\dots,k\}$, such that $\int fd\nu_\rho=0$ for any $\rho \in [0,1]$ we have
\begin{equation}
\label{ec.SG}
\int f^2 d\nu_\rho \leq \lambda_0 k^2 \sum_{x=1}^{k-1} \mc D_x(f).
\end{equation}
\end{proposition}

Now we explain how to use this proposition to estimate the norm $\|f\|_{-1,n}$. Let $f \in L^2(\nu_\rho)$ be such that $\int f d\nu_\rho=0$ for any $\rho \in [0,1]$. Assume that the support of $f$ is contained in $A=\{1,\dots,k\}$ and let $g \in L^2(\nu_\rho)$ be arbitrary.  Define $g_A= E[g|\mc F_A]$. Then $\<f,g\>_\rho = \<f,g_A\>_\rho$ and
\begin{align*}
\<g,-L_ng\>_\rho \geq \epsilon_0 n^2 \mc D(g) 
	\geq \epsilon_0 \sum_{x=1}^{k-1} \mc D_x(g) \geq \epsilon_0 n^2\sum_{x=1}^{k-1} \mc D_x(g_A). \\
\end{align*}
 
 Therefore, 
 \[
 \|f\|_{-1,n}^2 \leq  \sup_g\big\{ 2\<f,g\>_\rho -\epsilon_0 n^2 \sum_{x=1}^{k-1} \mc D_x(g)\big\},
 \]
 where now the supremum is over functions $g$ such that $\supp(g) \subseteq A$. We will use the spectral gap inequality \eqref{ec.SG} in order to obtain a lower bound for $\sum_k \mc D_x(g)$. But $g$ does not necessarily satisfies the hypothesis of Proposition~\ref{p6}. Let us define $\bar g = g- E[g|\eta^k(0)]$. Since $\int f d \nu_\rho = 0$ for any $\rho \in [0,1]$, we have $\<f,g\>_\rho =\<f,\bar g\>_\rho$. For $1 \leq x \leq k-1$,  the transformation $\eta \to \eta^{x,x+1}$ does not change the value of $\eta^k(0)$. Therefore, we also have $\mc D_x(g) = \mc D_x(\bar g)$. Now we can use the spectral gap inequality, since  Proposition \ref{p6} applies for $\bar g$. Using Proposition \ref{p6} we obtain that
 \[
 \|f\|_{-1,n}^2 \leq \sup_g\{ 2\<f,g\>_\rho -\frac{\epsilon_0 n^2}{\lambda_0 k^2} \<g,g\>_\rho\},
 \]
 where the supremum is over functions $g$ which are $\mc F_A$-measurable. This last supremum can be computed explicitly and it is equal to $\lambda_0 k^2 \<f,f\>_\rho/\epsilon_0 n^2$. Since the measure $\nu_\rho$ and the Dirichlet form $\mc D(f)$ are translation invariant, the same estimate holds whenever the support of $f$ is contained on an interval of size $k$. More relevant to our purposes is that this estimate is additive in the following sense. Let $f_1, f_2:\Omega \to \bb R$ be such that $\int f_i d\nu_\rho = 0$ for any $\rho \in [0,1]$ and $i=1,2$. Assume that $\supp(f_1) \subseteq \{1,\dots,k\}$ and $\supp(f_2) \subseteq \{k+1,\dots,k+l\}$ with $k,l \in \bb N$. Then
 \begin{align*}
 \sup_g\big\{2\<f_1+f_2,g\>_\rho - \<g,-L_ng\>_\rho\big\} 
 	&\leq \sup_g\big\{2\<f_1,g\>_\rho - \epsilon_0 n^2 \sum_{x=1}^{k-1} \mc D_x(g)\big\}\\
	&+ \sup_g\big\{2\<f_2,g\>_\rho - \epsilon_0 n^2 \sum_{x=k+1}^{k+l-1} \mc D_x(g)\big\},
 \end{align*}
 from where we conclude that
 \[
 \|f_1+f_2\|_{-1,n}^2 \leq \frac{\lambda_0 }{\epsilon_0 n^2} \big\{ k^2\<f_1,f_1\>_\rho + l^2\<f_2,f_2\>_\rho \big\}.
 \]
 
In other words, if the supports of $f_1$ and $f_2$ are contained in disjoint intervals, Kipnis-Varadhan estimate is additive. Let us rewrite this observation as a proposition.

\begin{proposition}
\label{p7}
Let $\{A_i = \{l_i+1,\dots,l_i+k_i\}; i \in I\}$ a collection of disjoint intervals in $\bb Z$. Let $f_i:\Omega \to \bb R$ be such that $\supp(f_i) \subseteq A_i$ and such that $\int f_i d\nu_\rho=0$ for any $\rho \in [0,1]$ and any $i$. There exists a constant $c_0=c_0(\epsilon_0)$ such that
\[
\bb E_n\Big[ \sup_{0 \leq t \leq T} \Big(\int_0^t \sum_{i \in I} f_i(\eta_s^n) ds \Big)^2\Big] 
	\leq \frac{c_0 T}{n^2} \sum_{i \in I} k_i^2 \<f_i,f_i\>_\rho.
\]
\end{proposition}

An important final step in the proof of hydrodynamic limits is the so-called {\em equivalence of ensembles}. In order to state this property, we need to introduce some notation. 
For $k \in \bb N$ and $x \in \bb Z$, we define
\[
\eta^k(x) = \frac{1}{k}\sum_{i=1}^k \eta(x+i), \text{ and}
\]
In other words, $\eta^k(x)$ is the density of particles in a box of size $k$ at the right of $x$.
Let $f: \Omega \to \bb R$ be a local function and assume that $\supp(f) \subseteq \{1,\dots,l\}$. For $k >l$ and $m \in \{0,1,\dots,k\}$, we define $\Omega_k = \{0,1\}^{\{1,\dots,k\}}$ and 
\[
\Omega_{k,m} = \big\{\eta \in \Omega_k; \sum_{x=1}^{k} \eta(x) =m\big\}.
\]

We define the measure $\nu_{k,m}$ as the uniform measure in $\Omega_{k,m}$. Notice that $\nu_{k,m}$ is also equal to the Bernoulli measure $\nu_\rho$ restricted to $\Omega_k$ and conditioned to the set $\Omega_{k,m}$.  We define the function
\[
\Psi(k,x) = \int f d\nu_{k,kx} = E[f| \eta^k(0) = x]
\]
for any $x$ of the form $m/k$, where $m \in \{0,1,\dots,k\}$. The function $\Psi(k,x)$ will play a fundamental role in the proof of Theorem \ref{t2}. Recall also the definition $\psi(\rho) = \int f d\nu_\rho$. For a local function $f$ such that its support is not necessarily contained on a set of the form $\{1,\dots,l\}$ there exist some positive numbers $l,x$ such that $\supp(\tau_x f) \in \{1,\dots,l\}$. Then we define $\Psi(k,x) = E[\tau_x f| \eta^k(0)=x]$. 

\begin{proposition}[Equivalence of ensembles]
\label{p8}
Let $f: \Omega \to \bb R$ be a local function. There exists a constant $c_{eq}=c_{eq}(f)$ such that 
\[
\big|\Psi(k,x) - \psi(x) + \frac{x(1-x)}{2k} \psi''(x)\big|Ê\leq \frac{c_{eq}}{k^2}
\]
for any $x \in \{0,1/k,\dots,1\}$ and any $k \in \bb N$.
\end{proposition}

This proposition is classical, and for the sake of completeness we present in the Appendix a simple proof for the particular case of Bernoulli product measures considered here. Notice that there is no mention to any particular density $\rho \in [0,1]$ in this proposition. The following proposition explains what the equivalence of ensembles says for a given fixed density $\rho \in [0,1]$ and for a function $f$ as in Theorem \ref{t2}.

\begin{proposition}
\label{p9}
Let $f$ be a local function and let $\rho \in [0,1]$ be fixed. Assume that $\psi(\rho) = \psi'(\rho)=0$. Then, for any $p \in \bb N$ there exists a constant $c_{eq}=c_{eq}(f,p)$ such that
\[
\int \Psi(k,\eta^k(0))^{2p} \nu_\rho(d\eta) \leq \frac{c_{eq}}{k^{2p}}
\]
and
\[
\int \Big(\Psi(k, \eta^k(0)) -\frac{\psi''(\rho)}{2}\Big\{\big(\eta^k(0)-\rho)^2-\frac{\chi(\rho)}{k}\Big\}\Big)^2 \nu_\rho(d \eta) \leq \frac{c_{eq}}{k^3}.
\]
\end{proposition}

\begin{proof}
The first inequality is a straight-forward consequence of Proposition~\ref{p8}. The second one follows from Proposition~\ref{p8} and a second-order Taylor expansion of $\psi(x)$ around $x=\rho$. Let  $f$ be a local function and let $\rho \in (0,1)$ be fixed. Then there exist bounded functions $\mc R^i: [0,1]\times [0,1] \to \bb R$, $i=1,2$ such that 
\begin{align*}
\Psi(k,x) 
	&= \psi(\rho) + \psi'(\rho)(x-\rho) +\frac{\psi''(\rho)}{2}\Big\{(x-\rho)^2-\frac{\chi(\rho)}{k}\Big\}\\
	&\quad + \mc R^1(\rho,x)(x-\rho)^3 +\mc R^2(\rho,x)\frac{x-\rho}{k}. 
\end{align*}
\end{proof}
\subsection{The multiscale analysis}
\label{s3.2}

In this section we prove Theorem~\ref{t2}. Roughly speaking, the proof has two steps. First we change the function $\tau_x f$ by a function of $\eta^k(x)$, for a suitable $k=k_0$. We call this step the {\em seed}. This step is analogous to the one-block estimate (see Section 5.4 of \cite{KL}). Then we define a sequence of scales $\{k_i\}_i$ and we inductively change $\eta^{k_i}(x)$ by $\eta^{k_{i+1}}(x)$. This last step is repeated until the final $k_i$ is equal to $\epsilon n$. A careful estimation of the error introduced at each approximation plus a proper choice of the sequence $\{k_i\}_i$ will prove the theorem.

It will be convenient to extend a little the definition of the field $\mc A_t^n(G)$. Let us recall that we are assuming that the local function $f:\Omega \to \bb R$ satisfies $\psi(\rho)=\psi'(\rho)=0$. For ease of notation, we assume that there is $l\in \bb N$ such that $\supp(f) \subseteq \{1,\dots,l\}$. Let us fix $n >l$ for now. 
Let  $G: \bb Z \to \bb R$ be fixed and define $H: \bb Z \to \bb R$ as $H_x = n(G_{x+1}-G_x)$. We assume that 
\begin{equation}
\label{ec11}
\|G\|_{1,n}^2 = \frac{1}{n} \sum_{x \in \bb Z} (H_x)^2 <+\infty.
\end{equation}
For this function $G$ we define $\mc A_t^n(G)$ in the obvious way:
\[
\mc A_t^n(G) = \int_0^t \sum_{x \in \bb Z} \tau_x f(\eta_s^n) H_xds.
\]
Of course, our previous definition corresponds to take $H_x = \nabla^n_x G$. Let $k>l$ be fixed. Later $k$ will be chosen as a function of $n$. Our first task is to introduce an average of the functions $\tau_x f$ over a box of size $k$. As we mentioned before, the variance of the field $\sum_x \tau_x f(\eta) H_x$ grows like $n$. Therefore, even at this preliminary stage we need to use the time evolution of the system.  Let us denote by $k \bb Z$ the sublattice $\{kx;x \in \bb Z\}$. For $x \in k\bb Z$, let us define
\begin{equation}
\label{ec9}
H_x^{k} = \frac{1}{k} \sum_{i=x+1}^{x+k} H_{i}.
\end{equation}

From Cauchy-Schwarz inequality,
\[
(H_x^k)^2 \leq 1/k \sum_{i=1}^k (H_{x+i})^2
\]
and we conclude that for any $k >0$,
\[
\sum_{x \in k \bb Z} (H_x^k)^2 \leq \frac{n}{k} \|G\|_{1,n}^2.
\]
In what follows, we will make repeated use of this elementary inequality, without further mention to it.  Let us define the rest $R_t^{0,n,1}(H)$ as
\begin{equation}
\label{ec4}
\begin{split}
R_t^{0,n,1}(H) 
	&=\mc A_t^n(G) - \int_0^t \sum_{x \in k\bb Z} \sum_{i=x+1}^{x+k} \tau_{i} f(\eta_s^n) H_x^{k} ds \\
	&= \int_0^t  \sum_{x \in k\bb Z} \sum_{i=x+1}^{x+k} \tau_{i} f(\eta_s^n) \big(H_{i}-H_x^{k}\big)ds.
\end{split}
\end{equation}

Notice that for each $x \in k\bb Z$ the sum
\begin{equation}
\label{ec5}
\sum_{i=x+1}^{x+k} \tau_{i} f(\eta) \big(H_i -H_x^{k}\big)
\end{equation}
has mean zero with respect to $\nu_\rho$ for any $\rho \in [0,1]$. Since $k >l$ , for $x,y \in k\bb Z$ such that $|x-y| >k$, the corresponding sums in \eqref{ec5} have supports contained on disjoint intervals of length at most $2k$. Therefore, we can split the sum on the right side of \eqref{ec4} into two pieces, each one of which satisfies the hypothesis of Proposition \ref{p7}. In this way we get the estimate
\begin{align*}
\bb E_n\big[\sup_{0\leq t\leq T} (R_t^{0,n,1})^2\big] 
	&\leq \frac{4c_0 T}{n^2} \sum_{x \in k\bb Z} 4 k^2 
	\int \Big(\sum_{i=x+1}^{x+k}  \tau_i f(\eta)\big(H_i-H_x^{k}\big)\Big)^2\nu_\rho(d\eta) \\
	& \leq \frac{16 c_0 T l k^2 \<f,f\>_\rho}{n^2} \sum_{x \in k \bb Z} \sum_{i=x+1}^{x+k} \big(H_i-H_x^{k}\big)^2\\
	&\leq \frac{C(f)C_{n,k}(G)Tk^2}{n},
\end{align*}
where $C(f)= 16c_0 l \<f,f\>_\rho$ is a constant that only depends on $f$
\footnote{From now on, we do not make explicit the dependence of $f$ of the various constants appearing. Therefore, constants like $C(\rho)$ may also depend on $f$.}
and 
\[
C_{n,k}(G) = \frac{1}{n} \sum_{x \in k\bb Z} \sum_{i=x+1}^{x+k} \big(H_i-H_x^{k}\big)^2.
\]

Notice that $C_{n,k}(G) \leq \|G\|_{1,n}^2$. We conclude that $R_t^{0,n,1}\to 0$ when $n \to \infty$ as soon as $k^2/n \to 0$. Let us define $k_0 = 2k$, where $k = k(n)$ is such that $k^2/n \to 0$ as $n \to \infty$. Up to here, we have written $\mc A_t^n(G)$ as
\begin{equation}
\label{ec6}
 \int_0^t \sum_{x \in k\bb Z} \sum_{i=1}^k \tau_{x+i} f(\eta_s^n) H_x^{k} ds
\end{equation}
plus a rest that vanishes in $L^2(\bb P_n)$ as $n \to \infty$. Since the support of $f$ can be bigger than a single point, the sums over each block of size $k$ do not have disjoint support. But since $k > l$, if we split the sum over $x \in k \bb Z$ into two alternated sums, one over $x \in 2k \bb Z$ and another one over $x+k \in 2k \bb Z$, the corresponding supports will be contained in disjoint intervals of size $2k$. Let us define 
\[
\mc A_t^{n,e} =\int_0^t \sum_{x \in 2k \bb Z} \sum_{i=x+1}^{x+k} \tau_i f(\eta_s^n) H_x^{k}ds,
\]
\[
\mc A_t^{n,o} =\int_0^t \sum_{x\in 2k \bb Z} \sum_{i=x+k+1}^{x+2k} \tau_i f(\eta_s^n) H_{x+k}^{k}ds.
\]

We will concentrate ourselves in $\mc A_t^{n,e}$, the computations for $\mc A_t^{n,o}$ being exactly the same. Notice that the support of the sum $\sum_i \tau_i f(\eta) H_x^{k}$ is contained in the interval $\{x+1,\dots,x+2k\}$. For $x \in \bb Z$ and $k,n \in \bb N$, define 
\[
\eta_s^{n,k}(x) = \frac{1}{k}\sum_{i=1}^k \eta_s^n(x+i).
\]
Let us recall the definition of $\Psi(k,x)$ and let us write
\[
\mc A_t^{n,e}
	= R_t^{0,n,e} + \int_0^t \sum_{x \in 2k \bb Z} k \Psi(2k; \eta_s^{n,2k}(x)) H_x^{k}ds,
\]
where
\[
R_t^{0,n,e} =  \int_0^t \sum_{x \in 2k \bb Z} \Big\{\sum_{i=x+1}^{x+k} \tau_i f(\eta_s^n) - k \Psi(2k,\eta_s^{n,2k}(x)) \Big\}H_x^{k}.
\]
The error term $R_t^{0,n,o}$ is defined in a similar way. Notice that $k\Psi(2k,\eta^{2k}(x))$ is equal to the conditional expectation of $\sum_i \tau_{x+i} f(\eta)$ on the corresponding box of size $2k$.
Therefore, in this integral each term of the sum over $2k \bb Z$ has mean zero with respect to each measure $\nu_\rho$ and we can use Proposition \ref{p7} to estimate the variance of $R_t^{0,n,e}$. Repeating the computations done to compute the variance of $R_t^{0,n,1}$ we see that
\[
\bb E_n\big[\sup_{0\leq t \leq T} (R_t^{0,n,e})^2\big] \leq \frac{4c_0 T l k^3 \<f,f\>_\rho}{n^2} \sum_{x \in 2k \bb Z} \big(H_x^{k}\big)^2 \leq \frac{C(\rho)\|G\|_{1,n}^2Tk^2}{n}.
\]
A similar estimate holds for $R_t^{0,n,o}(H)$. Let us define $R_t^{0,n} = R_t^{0,n,1}+R_t^{0,n,e}+R_t^{0,n,o}$. Recall the choice $k_0 =2k$. Putting the three estimates together and observing that for any $x \in 2k \bb Z$, $H_x^{k}+H_{x+k}^{k} = 2 H_x^{2k}$, we have proved that 
\[
\mc A_t^n(G) = \int_0^t \sum_{x \in k_0 \bb Z} \Psi(k_0,\eta_s^{n,k_0}(x)) H_x^{k_0} ds + R_t^{0,n}(H),\]
where
\begin{equation}
\label{ec7}
\bb E_n \big[\sup_{0\leq t \leq T} (R_t^{0,n})^2\big] \leq \frac{C(\rho) \|G\|_{1,n}^2 T k^2}{n}.
\end{equation}

This decomposition is what we call the {\em seed}. What this decomposition is telling us, is that we can replace the weighted averages of the functions $\tau_x f$ by a function of the density of particles on a ``not-too-big'' block of size $k_0$. The informed reader is invited to notice the parallel between this decomposition and the one-block estimate introduced in \cite{GPV}.

Now we are ready to start the multiscale argument. For $k \in \bb N$, let us define 
\[
\mc A_t^{n,k}(G) = \int_0^t \sum_{x\in k\bb Z} k \Psi(k,\eta_s^{n,k}(x)) H_x^{k} ds.
\]
We have the following estimate:

\begin{theorem}[Iterative bound]
\label{t3}
For any $k \in \bb N$ and any $H$ bounded and of compact support,
\[
\bb E_n\Big[\sup_{0\leq t \leq T} \big(\mc A_t^{n,2k}(G) -\mc A_t^{n,k}(G)\big)^2\Big] \leq 
	\frac{C(\rho) \|G\|_{1,n}^2 kT}{n}.
\]
\end{theorem}

\begin{proof}
Notice that for any $x \in 2k \bb Z$,
\[
E[\Psi(k,\eta^k(x))|\eta^{2k}(x)] = E[\Psi(k,\eta^k(x+k))|\eta^{2k}(x)] = \frac{1}{2} \Psi(2k,\eta^{2k}(x)).
\]
This is evident from the definition of $\Psi(k,\eta^k(0))$ as the conditional expectation of $f$ with respect to $\eta^k(0)$. Notice as well that $H_x^{k}+H_{x+k}^{k}=2H_x^{2k}$. Therefore, we can write
\begin{align*}
2k \Psi(2k,\eta^{2k}(x)) H_x^{2k} 
	&= E\big[k \Psi(k,\eta^{k}(x)) H_x^{k}\\
	&\quad +k \Psi(k,\eta^{k}(x+k)) H_{x+k}^{k}\big|\eta^{2k}(x)\big].
\end{align*}

We conclude that the functions
\[
F_x(\eta) = k \Psi(k,\eta^{k}(x)) H_x^{k}+k \Psi(k,\eta^{k}(x+k)) H_{x+k}^{k}-2k \Psi(2k,\eta^{2k}(x)) H_x^{2k}
\]
have mean zero with respect to each measure $\nu_\rho$. The support of the functions $F_x$ is contained on the interval $\{x+1,\dots,x+2k\}$. Therefore, the functions $F_x$ satisfy the hypothesis of Proposition~\ref{p7}. By the definition of $\psi(k,\eta^k(x))$ as a conditional expectation and by Proposition \ref{p9}, we see that 
\[
\<F_x,F_x\>_\rho \leq  c_{eq} \{(H_x^{k})^2+(H_{x+k}^{k})^2\}.
\]
Since $\mc A_t^{n,2k}(G)-\mc A_t^{n,k}(G)= \int_0^t \sum_x F_x(\eta_s^n)ds$, using Proposition~\ref{p7} we obtain the desired bound.
\end{proof}

For $i \in \bb N$ we define $k_i = 2^ik_0$. To avoid overcharged notation we write $\mc A_t^{n,i} = \mc A_t^{n,k_i}(G)$. By Theorem \ref{t3}, we have 
\[
\bb E_n\Big[\sup_{0\leq t \leq T}Ê\big(\mc A_t^{n,i+1} - \mc A_t^{n,i}\big)^2\Big] 
	\leq \frac {C(\rho) \|G\|_{1,n}^2 T 2^i k_0}{n}.
\]
Therefore, writing 
\[
\mc A_t^{n,m} -\mc A_t^{n,0} = \sum_{i=0}^{m-1} \big\{\mc A_t^{n,i+1} -\mc A_t^{n,i}\big\}
\]
and using Minkowski's inequality we see that
\begin{align*}
\bb E_n\Big[\sup_{0\leq t \leq T} \big(\mc A_t^{n,m} -\mc A_t^{n,0}\big)^2\Big]
	&\leq \Big( \sum_{i=0}^{m-1} \bb E_n \Big[\sup_{0 \leq t \leq T} \big( \mc A_t^{n,i+1} -\mc A_t^{n,i} \big)^2\Big]^{1/2} \Big)^2\\
	&\leq \Bigg(\sum_{i=0}^{m-1} \sqrt{\frac{C(\rho)\|G\|_{1,n}^2T k_0 2^i}{n}}\Bigg)^2\\
	&\leq \frac{C(\rho) \|G\|_{1,n}^2T k_0 2^m}{n(\sqrt 2 -1)^2}
	\leq \frac{C(\rho) \|G\|_{1,n}^2Tk_m }{n}.
\end{align*}

Notice that this last estimate only depends on $k_m$ and not on $k_0$.
Fix $\epsilon >0$ and take $m = \log(\epsilon n/k_0)$. Recall the estimate \eqref{ec7}. Choosing $k_0 = \sqrt{\epsilon n}$ and putting these two estimates together, we conclude that
\begin{equation}
\label{ec8}
\bb E_n\Big[\sup_{0 \leq t \leq T}Ê\big(\mc A_t^n(G) - \mc A_t^{n,\epsilon n}(G)\big)^2\Big] 
	\leq C(\rho) \|G\|_{1,n}^2 T \epsilon.
\end{equation}

This estimate basically finishes the proof of Theorem~\ref{t2}. Recall the definition of $\mc A_t^{n,\epsilon n}(G)$:
\[
\mc A_t^{n,\epsilon n}(G) 
	=\int_0^t \sum_{x\in \epsilon n\bb Z} \epsilon n \Psi(\epsilon n, \eta_s^{n,\epsilon n}(x)) H_x^{\epsilon n} ds.
\]
Let us define
\[
R_x^{n,\epsilon}(\eta) = \Psi(\epsilon n,\eta^{\epsilon n}(x))-\frac{\psi''(\rho)}{2}\Big\{\big(\eta^{\epsilon n}(x)-\rho\big)^2 -\frac{\chi(\rho)}{\epsilon n}\Big\}.
\]
We can rewrite $\mc A_t^{n,\epsilon n}(G)$ as
\begin{align*}
\mc A_t^{n,\epsilon n}(G) 
	&= \frac{\psi''(\rho)}{2} \int_0^t \sum_{x \in \epsilon n \bb Z} \epsilon n \Big\{\big( \eta_s^{n,\epsilon n}(x)-\rho\big)^2 - \frac{\chi(\rho)}{\epsilon n}\Big\}H_x^{\epsilon n} ds \\
	&+ \int_0^t \sum_{x \in \epsilon n \bb Z}Ê\epsilon n R_x^{n,\epsilon}(\eta_s^n) H_x^{\epsilon n} ds.
\end{align*}
Notice that the constant $\chi(\rho)/\epsilon n$ is not needed when $\sum_x H_x=0$.
By Proposition \ref{p9} we have 
\begin{equation}
\label{ec10}
\int \Big(  \sum_{x \in \epsilon n \bb Z}Ê\epsilon n R_x^{n,\epsilon}(\eta) H_x^{\epsilon n} \Big)^2\nu_\rho(d \eta) 
	\leq  \frac{c_{eq}}{\epsilon n} \sum_{x \in \epsilon n \bb Z} (H_x^{\epsilon n})^2 
	\leq \frac{c_{eq}\|G\|_{1,n}^2}{\epsilon^2 n}.
\end{equation}
Notice that
\[
\mc Y_t^n(i_\epsilon(x/n))^2 = n\big(\eta_t^{n,\epsilon n}(x)-\rho\big)^2.
\]
Putting estimate \eqref{ec10} together with estimate \eqref{ec8} and choosing $H_x = \nabla_x^n G$, we conclude that
\begin{multline*}
\bb E_n\Big[\Big(\mc A_t^n(G) -\frac{\psi''(\rho)}{2}\int_0^t \sum_{x \in \bb \epsilon Z} \big(G(x+\epsilon)-G(x)\big)\mc Y_s^n(i_\epsilon( x))^2 ds\Big)^2\Big]\leq \\
	\leq C(\rho) \|G\|_{1,n}^2\big\{ T \epsilon + \frac{T^2}{\epsilon^2 n}\big\},
\end{multline*}
which proves Theorem~\ref{t2}. Here we make two remarks about this result.  First, we did not only obtain the convergence result stated in Theorem \ref{t2}, but we also obtained a good control on the rate of convergence. This point will be important in Section \ref{s2.2}, more precisely to prove that limit points of $\mc Y_t^n$ will be energy solutions of equation \eqref{ec.KPZ}. And second, this estimate shows that Theorem \ref{t2} holds in general for any function $G: \bb R \to \bb R$ such that  $\|G\|_{1,n}^2$ is finite for any $n$ and uniformly bounded in $n$.

\section{Proof of the equilibrium fluctuations}
\label{s2}
The proof of Theorem~\ref{t1} is based on the study of some martingales associated to the process $\mc Y_t^n$. Let $F: \Omega \to \bb R$ be a function on the domain of the generator $L_n$. Dynkin's formula says that the process
\[
M_t^{F,n} = F(\eta_t^n) -F(\eta_0^n) -\int_0^t L_n F(\eta_s^n) ds
\]
is a martingale with respect to the natural filtration associated to the process $\eta_t^n$. If $F^2$ also belongs to the domain of the generator $L_n$, the quadratic variation of $M_t^{F,n}$ is given by
\[
\<M_t^{F,n}\> = \int_0^t \big\{ÊL_n F(\eta_s^n)^2 - 2 F(\eta_s^n) L_n F(\eta_s^n)\big\}ds.
\]
We will use this formula for $F(\eta_t^n) = \mc Y_t^n(G)$ for $G \in \mc S(\bb R)$. After some calculations, we see that
\begin{equation}
\label{ec.M1}
\begin{split}
M_t^n(G)
	&= \mc Y_t^n(G) -\mc Y_0^n(G) -\int_0^t \!\!\frac{1}{2\sqrt n} \sum_{x \in \bb Z}Ê\tau_x h(\eta_s^n) \Delta_x^nG ds - \int_0^t \! \sum_{x \in \bb Z}Ê\tau_x f(\eta_s^n) \nabla_x^n Gds
\end{split}
\end{equation}
is a martingale, where $f(\eta) = \frac{1}{2}a c(\eta) (\eta(1)-\eta(0))^2$ and
\[
\Delta_x^n G = n^2\Big\{G\Big(\frac{x+1}{n}\Big) +G\Big(\frac{x-1}{n}\Big)-2G\Big(\frac{x}{n}\Big)\Big\}
\]
is a discrete approximation of $G''(x/n)$. The quadratic variation of $M_t^{n}(G)$ is given by
\[
\<M_t^{n}(G)\> = \int_0^t \frac{1}{n} \sum_{x \in \bb Z}Ê\tau_x g_n(\eta_s^n) \big(\nabla_x^n G\big)^2ds,
\]
where $g_n(\eta) = \{p_n \eta(0)(1-\eta(1))+q_n\eta(1)(1-\eta(0))\}c(\eta)$. Looking at formula \eqref{ec.M1}, we can write $\mc Y_t^n(G)$ as the sum of four terms: the initial value $\mc Y_0^n(G)$, the martingale $M_t^n(G)$ and two integral terms. Starting from this decomposition, we will prove in Sect.~\ref{s2.1} that the sequence of processes $\{\{\mc Y_t^n; t \geq 0\}; n \in \bb N\}$ is tight with respect to the uniform topology in $\mc D([0,\infty), \mc S'(\bb R))$; and then in Sect.~\ref{s2.2} we will 
prove that any limit point of the sequence $\{\{\mc Y_t^n; t \geq 0\}; n \in \bb N\}$ is a weak solution of the KPZ equation.

\subsection{Tightness of the density field}
\label{s2.1}
In this section we prove tightness of the sequence  $\{\{\mc Y_t^n; t \geq 0\}; n \in \bb N\}$. As usual, to avoid uninteresting topology issues, we fix $T>0$ and we consider the processes $\mc Y_t^n(G)$ restricted to the interval $[0,T]$. We will use Mitoma's criterion \cite{Mit}, which now we describe. Let $X$ be a complete, separable metric space with metric $d: X \times X \to [0,\infty)$ and let $\{P_n;n \in \bb N\}$ be a sequence of probability measures in $\mc D([0,T], X)$. We say that $\{P_n; n\in \bb N\}$ is {\em $\mc C$-tight} if $\{P_n;n \in \bb N\}$ is tight with respect to the uniform measure in $\mc D([0,T],X)$. Let $P$ be a probability measure on $\mc D([0,T],\mc S'(\bb R))$. For $G \in \mc S(\bb R)$ we denote by $P^G$ the probability measure in $\mc D([0,T],\bb R)$ defined by
\[
P^G(A) = P(x_\cdot(G) \in A),
\] 
where $x_\cdot(G) \in \mc D([0,T],\bb R)$ is given by $x_\cdot(G) = \{x_t(G); t \in [0,T]\}$. Mitoma's criterion says the following:
\begin{proposition}
\label{p10}
Let $\{P_n;n \in \bb N\}$ be a sequence of probability measures in $\mc D([0,T],\mc S'(\bb R))$. The sequence $\{P_n; n \in \bb N\}$ is $\mc C$-tight if and only if $\{P_n^G; n\in \bb N\}$ is $\mc C$-tight for any $G \in \mc S(\bb R)$.
\end{proposition}

As a consequence of this criterion, in order to prove $\mc C$-tightness of the sequence $\{\{\mc Y_t^n; t \in [0,T]\};n \in \bb N\}$, it is enough to prove $\mc C$-tightness of the sequence  $\{\{\mc Y_t^n(G); t \in [0,T]\};n \in \bb N\}$ for any function $G \in \mc S(\bb R)$. It is enough to prove tightness for each one of the four terms appearing in \eqref{ec.M1}. First notice that for any $\theta \in \bb R$,
\begin{align*}
\bb E_n\big[ \exp\{ i \theta \mc Y_t^n(G)\}]
	&= \prod_{x \in \bb Z}Ê\bb E_n\big[ \exp\big\{\frac{i \theta}{\sqrt n} (\eta_0^n(x)-\rho)G(x/n\big\}\big]\\
	&=\prod_{x \in \bb Z} \Big\{Ê1-\frac{\theta^2}{2n}\chi(\rho)G(x/n)^2 +\frac{R_x^n}{6n^{3/2}}\Big\},
\end{align*}
where $|R_x^n| \leq |G(x/n)|^3$. Therefore, $\mc Y_0^n(G)$ conveges in distribution to a normal random variable with mean zero and variance $\chi(\rho)\int G(x)^2dx$. We conclude that the sequence $\{\mc Y_0^n(G);n \in \bb N\}$ is tight (since it is convergent). The following proposition says that we do not need to prove $\mc C$-tightness: it is enough to prove tightness with respect to the $J_1$-Skorohod topology of $\mc D([0,T],\bb R)$.

\begin{proposition}[\cite{FPV}]
\label{p11}
Let $\{P_n; n \in \bb N\}$ be a sequence of probability measures on $\mc D([0,T],X)$. assume that
\begin{itemize}
\item[i)] The sequence $\{P_n; n \in \bb N\}$ is tight with respect to the $J_1$-Skorohod topology of $\mc D([0,T],X)$,

\item[ii)] For any $A>0$,
\[
\lim_{n \to \infty} P_n\big(\sup_{0 \leq t \leq T} |x(t)-x(t-)| \geq A\big) =0.
\]
\end{itemize}
Then $\{P_n; n \in \bb N\}$ is $\mc C$-tight.
\end{proposition}

In other words, this proposition tells us that if the jumps of $\mc Y_t^n(G)$ are getting smaller and smaller with $n$, then $J_1$-tightness and $\mc C$-tightness are equivalent. In our case 
\[
\sup_{0 \leq t \leq T}Ê\big| \mc Y_t^n(G)-\mc Y_{t-}^n(G)\big| \leq \frac{\|G\|_\infty}{\sqrt n},
\]
and it is enough to prove tightness of $\{\{\mc Y_t^n(G); t \in [0,T]\};n \in \bb N\}$ with respect to the $J_1$-Skorohod topology in $\mc D([0,T],\bb R)$. Notice that the process $M_t^n(G)$ also satisfies condition ii) of Proposition~\ref{p11}. The following criterion, known as Aldous' criterion, is very effective to prove tightness with respect to the $J_1$-Skorohod topology. 

\begin{proposition}[Aldous' criterion \cite{Ald}]
\label{p12}
Let $\{P_n;n \in \bb N\}$ be a sequence of probability measures on $\mc D([0,T], X)$. Let us assume that:
\item[i)] For any $t \in [0,T]$ and any $\varepsilon >0$ there is a compact set $K=K(t,\varepsilon)$ such that 
\[
\sup_{n \in \bb N}ÊP_n(x(t) \in K) \leq \varepsilon,
\]
\item[ii)] for any $\varepsilon >0$,
\[
\lim_{\delta \to 0} \limsup_{n \to \infty} \sup_{0 \leq \gamma \leq \delta}\sup_{\tau \in \mc T_T} P_n(d(x(\tau+\gamma,x))>\varepsilon) =0,
\]
where $\mc T_T$ is the set of stopping times bounded by $T$. Here we define $x(\tau+\gamma) = x(\tau)$ if $t+\gamma >T$. 
\end{proposition}

Now we turn into the tightness of $\<M_t^n(G)\>$. Let us recall the definition of the seminorm $\|G\|_{1,n}$ given in Section \ref{s3.2}. For a function $G: \bb R \to \bb R$, we have
\[
\|G\|_{1,n}^2 = \frac{1}{n} \sum_{x \in \bb Z} \big(\nabla_x^n G\big)^2.
\]
Notice that for functions $G$ such that $G' \in \mc S(\bb R)$, $\|G\|_{1,n}^2$ converges to $\int G'(x)^2 dx$ when $n$ goes to infinity. Since the number of particles per site is at most equal to 1, we have the simple bound 
\begin{equation}
\label{ec11.1}
\big|\<M_t^n(G)\> -\<M_s^n(G)\>\big| \leq  \epsilon_0^{-1}\|G\|_{1,n}^2|t-s|,
\end{equation}
valid for any $s,t$ (even random). Therefore,
\begin{align*}
\bb P_n\big(|M_{\tau+\gamma}^n(G)-M_\tau^n(G)|Ê>\varepsilon\big)
	&\leq \frac{1}{\varepsilon^2} \bb E_n\big[ \big(M_{\tau+\gamma}^n(G)-M_\tau^n(G)\big)^2\big]\\
	&\leq \frac{1}{\varepsilon^2} \bb E_n\big[ \<M_{\tau+\gamma}^n(G)\>-\<M_\tau^n(G)\>\big]\\
	& \leq \frac{\gamma \epsilon_0^{-1}\|G\|_{1,n}^2}{\varepsilon^2}.
\end{align*}
This proves condition ii) of Proposition~\ref{p12} for the martingales $M_t^n(G)$. Condition i) follows from the fact that $\bb E_n[M_t^n(G)^2]$ is uniformly bounded for $t \in [0,T]$, $n \in \bb N$.

Tightness of the integral term 
\[
\int_0^t \frac{1}{2\sqrt n}\sum_{x \in \bb Z} \tau_x h(\eta_s^n) \Delta_x^n G ds
 \]
 follows from the fact that there exists a constant $C=C(G,\rho)$ such that
 \[
 \bb E_n\Big[ \Big(\frac{1}{2\sqrt n} \sum_{x \in \bb Z} \tau_x h(\eta_t^n) \Delta_x^n G\Big)^2\Big] \leq C(G,\rho)
 \]
 for any $t \in [0,T]$ and any $n \in \bb N$ (see \cite{FPV} or \cite{KL} for more details). We are only left to prove tightness for
 \[
 \mc A_t^n(G) = \int_0^t \sum_{x \in \bb Z} \tau_x f(\eta_s^n) \nabla_x^n G ds.
 \]
It turns out that this term is the most difficult to analyze, and here we use the whole power of Theorem \ref{t2}. Notice that even to prove that $\mc A_t^n(G)$ is in $L^2(\bb P_n)$ seems hard to prove at first glance. From estimate \eqref{ec8} we see that
\begin{equation}
\label{ec21}
\bb E_n\big[ \big( \mc A_t^n(G) -\mc A_t^{n,\epsilon n}(G)\big)^2\big] \leq  C_n(\rho,G) t\epsilon.
\end{equation}
The constant $C_n(\rho,G)$ can be chosen as equal to $C(\rho) \|G\|_{1,n}^2$. At this part of the argument, the only important point is that $C_n(\rho,G)$ is uniformly bounded on $n$.
Take $H_x = \nabla_x^n G$ and define $H_x^k$ as in \eqref{ec9}. Using Cauchy-Schwarz inequality plus the invariance of $\nu_\rho$ under the dynamics, we obtain the estimate
\begin{equation}
\label{ec14}
\bb E_n\big[\mc A_t^{n,\epsilon n}(G)^2\big] \leq \frac{C_n(\rho,G)t^2}{\epsilon}.
\end{equation}
Putting these two estimates together we conclude that
 \begin{equation}
 \label{ec12}
 \bb E_n\big[ \mc A_t^{n}(G)^2\big] \leq C_n(\rho,G) \Big\{ t \epsilon +\frac{t^2}{\epsilon}\Big\}.
 \end{equation} 
Since the process $\eta_t^n$ is stationary, the same bound holds for $\bb E_n[(\mc A_{t+s}^n(G)-\mc A_s^n(G))^2]$ for any $s \in [0,T-t]$. Moreover, we can choose the constant $\epsilon$ in a convenient way. However, there is a small constraint on $\epsilon$. At the beginning of the multiscale analysis in Section \ref{s3.2} we took $k>l$. This choice, plus some parity considerations impose the restriction $\epsilon > 2l/n$. Taking $\epsilon = |t-s|^{1/2}$ we obtain that for any $s,t \in [0,T]$ such that $|t-s| \geq 4 l^2/n^2$ and any $n \in \bb N$,
 \begin{equation}
 \label{ec13}
 \bb E_n\big[\big(\mc A_t^n(G) -\mc A_s^n(G)\big)^2\big] \leq C(\rho,G) |t-s|^{3/2}.
 \end{equation}
If $|t-s| < 4l^2 /n^2$, we can simply use Cauchy-Schwarz inequality plus the sta\-tion\-arity of $\eta_t^n$ to obtain the bound
\begin{align*}
\bb E_n\big[ \big(\mc A_t^n(G) -\mc A_s^n(G)\big)^2\big] 
	&\leq C(\rho,G) |t-s|^2 n\\
	&\leq C(\rho,G)  2 l |t-s|^{3/2}.
\end{align*}

Therefore,  taking a bigger constant if needed, \eqref{ec13} holds for any $s,t \in [0,T]$ and any $n \in \bb N$. Notice that there is a constant $C(\rho)$ such that this estimate is still valid if we replace $C(\rho,G)$ by $C(\rho) \|G\|_{1,n}^2$.
Let us recall Kolmogorov-Prohorov-Centsov criterion for tightness:

\begin{proposition}[Kolmogorov-Prohorov-Centsov]
\label{p13}
Let $\{\{x_n(t); t\in [0,T]\}, n \in \bb N\}$ be a sequence of continuous processes in $\bb R$. Let us assume that there exist positive constants $\alpha, \beta, K$ such that
\[
E[|x_n(t)-x_n(s)|^\alpha] \leq K|t-s|^{1+\beta}
\]
for any $s,t \in [0,T]$ and any $n \in \bb N$. Assume as well that $\{x_n(0);n \in \bb N\}$ is tight. Then the sequence $\{\{x_n(t); t \in [0,T]\}, n \in \bb N\}$ is tight. Moreover, for any $\gamma <\alpha/\beta$, any limit point of $\{\{x_n(t); t\in [0,T]\}, n \in \bb N\}$ is almost-surely H\"older-continuous of index $\gamma$.
\end{proposition}

Invoking this proposition, we conclude that $\mc A_t^n(G)$ is tight. This finishes the proof of tightness of the sequence $\{\{\mc Y_t^n(G); t \in [0,T]\}, n\in \bb N\}$ and in consequence of $\{\{\mc Y_t^n; t \in [0,T]\}, n\in \bb N\}$.

Notice that this proposition also gives information about the limit points of $\mc A_t^n(G)$. If $\{\mc A_t; t\in[0,T]\}$ is such a limit point, then $\mc A_t^n(G)$ is H\"older-continuous of index $\gamma$ for any $\gamma<1/4$.

\subsection{Limit points of the density field}
\label{s2.2}
In this section we finish the proof of Theorem~\ref{t1} by showing that any limit point of $\mc Y_t^n$ is a weak solution of the KPZ equation. Later we also obtain some additional properties of the limiting points, with the aim of obtaining a uniqueness result for such solutions.

In Sect.~\ref{s2.1} we have showed that the sequence of processes $\{\mc Y_t^n ;n \in \bb N\}$ is $\mc C$-tight in $\mc D([0,T],\mc S'(\bb R))$. Let $\{\mc Y_t;t \in [0,T]\}$ be a limit point of $\mc Y_t^n$. From now on and up to the end of this section we adopt the following abuse of notation: $n$ will denote a subsequence for which $\mc Y_t^n$ converges to $\mc Y_t$. Let us recall the decomposition \eqref{ec.M1}:
\begin{equation}
\label{ec.M1'}
\mc Y_t^n(G) = \mc Y_0^n(G) + \mc I_t^n(G) + \mc A_t^n(G) + M_t^n(G),
\end{equation}
where for ease of notation we have written
\[
\mc I_t^n(G) = \int_0^t \frac{1}{2\sqrt n} \sum_{x \in \bb Z} \tau_x h(\eta_s^n) \Delta_x^n G ds.
\]

We can choose the subsequence $n$ in such a way that the processes $M_t^n$ and $\mc A_t^n$  have limits $M_t$, $\mc A_t$, well defined as processes in $\mc C([0,T],\mc S'(\bb R))$. By Proposition~\ref{p4}, the process $\mc I_t^n$ converges to $ \frac{1}{2}\varphi'(\rho)\int_0^t \mc Y_s ds$. Our first task is to prove that $M_t(G)$ is a continuous martingale for any $G \in \mc S(\bb R)$. Therefore, we need to prove that the martingale property is preserved by passing to the limit. This is not true in general under the only assumption of convergence in distribution. Therefore, we need an extra argument. 
The extra required property is {\em uniform integrability}. For the reader's convenience, we recall here some elements of uniform integrability which will be needed. 

We say that a sequence of random variables $\{X_n; n \in \bb N\}$ is uniformly integrable if
\[
\lim_{M \to \infty} \sup_{n \in \bb N} \int |X_n| \mathbf{1}(|X_n| \geq M) dP =0.
\]
The simplest criterion to verify the uniform integrability of a given sequence $\{X_n;n \in \bb N\}$ is a moment bound:

\begin{proposition}
\label{p14}
If there exists $p>1$ such that $\sup_n E|X_n|^p <+\infty$, then $\{X_n;n \in \bb N\}$ is uniformly integrable.
\end{proposition}

And a simple criterion to show that a limit of martingales is also a martingale is the following.

\begin{proposition}
\label{p15}
Let $\{\{\mc M_t^n, t \in [0,T]\};n \in \bb N\}$ be a sequence of martingales converging in distribution to some limit process $\mc M_t$. If for any fixed time  $t$ the sequence $\{\mc M_t^n; n \in \bb N\}$ is uniformly integrable, then the process $\mc M_t$ is a martingale.
\end{proposition}

Let us go back to our problem. Since the number of particles per site is bounded by $1$, the quadratic variation of $M_t^n(G)$ satisfies the deterministic bound $\<M_t^n(G)\> \leq \epsilon^{-1}_0 t \|G\|_{1,n}^2$. Therefore, $\bb E_n[M_t^n(G)^2] = \bb E_n[\<M_t^n(G)\>]$ is uniformly bounded, which proves that $M_t(G)$ is a martingale.

Thanks to Theorem~\ref{t3}, the process $\mc A_t$ satisfies
\[
\mc A_t(G) = \lim_{\epsilon \to 0} \int_0^t \int_{\bb R} \mc Y_s(i_\epsilon(x))^2 \frac{G(x+\epsilon)- G(x)}{\epsilon}dx ds.
\]
This form is slightly different to the conclusion of Theorem~\ref{t3}, where only a discrete summation in space is present. Considering sublattices of $\bb Z$ with different relative positions (something of the form $k \bb Z +l$), we can pass from a discrete spatial sum to the integral above. One important consequence of Theorem~\ref{t3} is that the process $\mc Y_t$ is such that the limit above is well defined, a property that does not hold for any process in $\mc D([0,T],\mc S'(\bb R))$. 

Estimating the variance of $\<M_t^n(G)\>$ using Cauchy-Schwarz inequality, we see that $\<M_t^n(G)\>$ converges in $L^2(\bb P_n)$ to $\chi(\rho)\varphi'(\rho) t \int G'(x)^2 dx$. In order to finish the proof of Theorem~\ref{t1} we are only left to prove that the quadratic variation of $M_t(G)$ is equal to this limit. This is equivalent to prove that $M_t(G)^2 - \chi(\rho)\varphi'(\rho) t \|G\|_1^2$ is a martingale. We already know that $M_t^n(G)^2 - \<M_t^n(G)\>$ is a martingale. Therefore, by Proposition~\ref{p15} we only need to show that $M_t^n(G)^2-\<M_t^n(G)\>$ is uniformly integrable. It is immediate to show that the second moments of $\<M_t^n(G)\>$ are uniformly bounded in $n$, as well as the fourth moments of $\mc Y_t^n(G)$ and $\mc I_t^n(G)$. Therefore, we only need to show that $\mc A_t^n(G)^2$ is uniformly integrable. In the course of the proof of Theorem~\ref{t3} we obtained a uniform bound on the $L^2$ norm of $\mc A_t^n(G)$, which is not enough to prove uniform integrability. A major step of the proof was obtained using Kipnis-Varadhan estimate. But Kipnis-Varadhan estimate does not lead to any useful information about moments of order higher than $2$, unless we impose restrictive additional hypothesis.   Therefore, we need another idea. A sequence bounded in $L^1$ is not uniformly integrable, but a sequence converging to 0 in $L^1$ {\em is} uniformly integrable. One more time we will use the decomposition 
\[
\mc A_t^n(G) = \big\{\mc A_t^n(G) - \mc A_t^{n,\epsilon n}(G)\big\} + \mc A_t^{n,\epsilon n}(G).
\]

In \eqref{ec14} we estimated the $L^2$-norm of the last term using Cauchy-Schwarz inequality. Using Proposition \ref{p9} with $p=2$, the same computations (replacing Cauchy-Schwarz inequality by H\"older inequality) give the bound
\[
\bb E_n\big[Ê\mc A_t^{n,\epsilon n}(G)^4\big] \leq \frac{C(G,\rho) t^4}{\epsilon^2}
\]
for any $\epsilon \leq 1$, where $C(G,\rho)$ is a constant who does not depend on $n$, $\epsilon$ or $t$. Therefore, $\mc A_t^n(G)$ is the sum of two terms, the first one (which is $\mc A_t^n(G) -\mc A_t^{n,\epsilon n}(G)$) has second moment bounded by $C(G,\rho) \epsilon t$ and the second one (which is $\mc A_t^{n,\epsilon n}(G)$) has fourth moment bounded by $C(G,\rho) t^4/\epsilon^2$ for any $\epsilon \leq 1$. From this observation and a good choice for $\epsilon$, the uniform integrability of $\mc A_t^n(G)^2$ follows: for $t \in [0,T]$ we have
\begin{align*}
\bb E_n\big[ \mc A_t^n(G)^2 \mathbf{1}(|\mc A_t^n(G)| \geq M)\big]
	 &\leq 2  \bb E_n\big[(\mc A_t^n(G) -\mc A_t^{n,\epsilon n}(G)^2\big]\\
	&\quad + 2\bb E_n \big[Ê\mc A_t^{n,\epsilon n}(G)^2 \mathbf{1}(|\mc A_t^n(G)| \geq M)\big]\\
	&\leq C \epsilon t + 2 \bb E_n \big[Ê\mc A_t^{n,\epsilon n}(G)^4\big]^{1/2}  \frac{\bb E_n \big[ \mc A_t^n(G)^2\big]^{1/2}}{M^{1/2}}\\
	&\leq C \Big\{ \epsilon +\frac{1}{\epsilon M^{1/2}}\Big\},
\end{align*}
 where $C$ is a constant who depends only on $G$, $\rho$ and $T$. Choosing $\epsilon = M^{-1/4}$  (actually $\epsilon = M^{-\delta}$ for any $\delta \in (0,1/2)$ would suffice) we conclude that $\mc A_t^n(G)$ is uniformly integrable. This proves that $M_t(G)$ is a martingale for any $G \in \mc S(\bb R)$.
 
 Up to here, we have proved that the process $\{\mc Y_t; t \in [0,T]\}$ is a weak stationary solution of the KPZ equation \eqref{ec.KPZ}. It only remains to prove that $\{\mc Y_t, t \in [0,T]\}$ is also an energy solution. This is not difficult to prove, since upper moment bounds are preserved by convergence in distribution. The bound
 \[
  E[\mc I_{s,t}(G)^2] \leq K(t-s)\|G\|_1^2
 \]
follows from the energy estimate stated in Proposition \ref{p16}Êbelow and stationarity. And the bound 
\[
\bb E\big[\big(\mc A_{s,t}(G) -\mc A_{s,t}^\epsilon(G)\big)^2\big] \leq C(\rho \epsilon t \|G\|_1^2
\]
follows from \eqref{ec21}.

Corollary \ref{cor1} is follows at once from \eqref{ec13} and Proposition \ref{p13}. And Theorem \ref{t2} follows by taking $c \equiv 1$ and identifying the process $\{\mc Y_t; t \in [0,T]\}$ with the Cole-Hopf solution of \eqref{ec.KPZ} obtained by Bertini and Giacomin.
\section{Current and height fluctuations}
\label{s4.0}
 In this section we prove the results stated in Section \ref{s1.4}. The idea is to exploit the formal relation $n^{-1/2} (J_t^n(0)-\bb E_n[J_t^n(0)]) = \mc Y_t^{n,*}(H_0)$, where the field $\mc Y_t^{n,*}$ is defined by
 \[
 \mc Y_t^{n,*}(G) = \mc Y_t^n(G) - \mc Y_0^n(G).
 \]
 
Clearly $H_0 \notin \mc S(\bb R)$ and therefore this relation must be justified. We consider the cut-off functions $G_l:\bb R \to \bb R$ defined by $G_l(x) = H_0(x)(1-x/l)^+$. This idea is due originally to Rost and Vares \cite{RV} and here we closely follow the exposition of \cite{JL}. For each $x \in \bb Z$, the process $J_t^n(x)$ is a compound Poisson process of compensator $j_{x,x+1}^n(\eta_s^n)-j_{x+1,x}^n(\eta_s^n)$, where
\[
j_{x,x+1}^n(\eta) = n^2 p_n c_x(\eta)\eta(x)(1-\eta(x+1))
\]
and
\[
j_{x+1,x}^n(\eta) = n^2 q_n c_x(\eta) \eta(x+1)(1-\eta(x)).
\]
Therefore, the process
\[
J_t^n(x) - \int_0^t \big\{ j_{x,x+1}^n(\eta_s^n) -j_{x+1,x}^n(\eta_s^n) \big\}ds
\]
is a martingale of quadratic variation
\[
\int_0^t \big\{j_{x,x+1}^n(\eta_s^n) + j_{x+1,x}^n(\eta_s^n) \big\} ds.
\]
Moreover, these martingales are mutually orthogonal for different values of $x$. Using the continuity equation, we see that
\begin{equation}
\label{ec4.1}
\theta_t^n(0) = \mc Y_t^{n,*}(G_l) + \mc A_t^n(H_0-G_l) + M_t^n(H_0-G_l) + \mc I_t^n(H_0-G_l).
\end{equation}

Notice that although $H_0-G_l$ is not integrable, $\nabla_x^n (H_0-G_l)$ has bounded support. Therefore, the three fields $\mc A_t^n(H_0-G_l)$, $M_t^n(H_0-G_l)$ and $\mc I_t^n(H_0-G_l)$ are well defined. A key observation here is that $G_l$ has a small energy: $\|G_l\|_{1,n}^2 = l^{-1}$. According to estimates \eqref{ec11.1}, \eqref{ec13}, we see that there exists a constant $K = K(\rho,T)$ such that
\begin{equation}
\label{ec4.2}
\begin{split}
\bb E_n\big[ \mc A_t^n(H_0-G_l)^2\big] 
	&\leq K l^{-1}\\
\bb E_n\big[ M_t^n(H_0-G_l)^2\big] 
	&\leq K l^{-1}.
\end{split}
\end{equation}

We also know the dependence of $K$ in $T$, but at this point the dependence will not be important.
The following proposition, known as the {\em energy estimate} tells us that $\mc I_t^n(H_0-G_l)$ satisfies the same estimate.

\begin{proposition}
\label{p16}
There exists a constant $K= K(\rho)$ such that
\[
\bb E_n\big[\mc I_t^n(G)^2 \big] \leq K t \|G\|_{1,n}^2
\]
for any $t \geq 0$ and any $G: \bb R \to \bb R$ such that $\|G\|_{1,n}<+\infty$.
\end{proposition}

This proposition follows easily from Kipnis-Varadhan inequality and corresponds to Lemma 4.2 of \cite{Qua} for example. We have the following lemma.

\begin{lemma}
\label{l1} There exists a constant $K= K(\rho,T)$ such that
\begin{equation}
\label{ec4.3}
\sup_{n \in \bb N} \bb E_n\big[\big( \theta_t^n(0) - \mc Y_t^{n,*}(G_l)\big)^2\big] \leq K l^{-1},
\end{equation}
\begin{equation}
\label{ec4.4}
\sup_{m \geq l} \sup_{n \in \bb N} \bb E_n\big[\big( \mc Y_t^{n,*}(G_m)- \mc Y_t^{n,*}(G_l)\big)^2\big] \leq K l^{-1}.
\end{equation}
\end{lemma}
The proof of \eqref{ec4.3} follows at once from \eqref{ec4.1}, Proposition \ref{p16} and inequalities \eqref{ec4.2} and we leave it to the reader. 
Estimate \eqref{ec4.4}  follows from the identity
\[
\mc Y_t^{n,*}(G_m) -\mc Y_t^{n,*}(G_l) = \mc A_t^n(G_m-G_l) + M_t^n(G_m-G_l) + \mc I_t^n(G_m-G_l).
\]

\begin{proof}[Proof of Theorem \ref{t4}]
Given Lemma \ref{l1}, the proof follows as in the articles \cite{Gon} and \cite{JL}.
Let $\{\mc Y_t; t \in [0,T]\}$ be a limit point of $\{\{\mc Y_t^n; t\in [0,T]\}; n \in \bb N\}$ and let $\bb P$ be the distribution of $\{\mc Y_t; t\in [0,T]\}$. Let $t >0$ be fixed. Recall the convention about the subsequence $n$. Approximating $G_l$ in $L^2(\bb R)$ by functions in $\mc S(\bb R)$, we see that $\mc Y_t(G_l)$ is well defined. Taking further subsequences if necessary, $\mc Y_t^n(G_l)$ converges in distribution to $\mc Y_t(G_l)$. Therefore, \eqref{ec4.4} is also satisfied by $\mc Y_t(G_l)$. In particular, $\{\mc Y_t(G_l); l \in \bb N\}$ is a Cauchy sequence in $L^2(\bb P)$ and there exists a random variable $\mc J_t(0)$ such that $\mc Y_t(G_l)$ converges to $\mc J_t(0)$ when $n \to \infty$. And from \eqref{ec4.3}, we conclude that $\theta_t^n(0)$ converges in distribution to $\mc J_t(0)$ as $n \to \infty$. Convergence of finite-dimensional distributions follows in the same way.
\end{proof} 

\subsection{Convergence of height fluctuations}
We start showing that $\theta_t^n(x)$ is a well-defined field in $\mc S'(\bb R)$. Looking at the definition \eqref{ec.height} , we see that the field $\theta_t^n$ has the following representation:
\[
\theta_t^n(x) = \theta_t^n(0) - \mc Y_t^{n}(\mathbf{1}_{(0,x]}).
\]

In particular, $\theta_t^n(x)$ can be written as the sum of a real-values process independent of $x$ and a two-sided Brownian motion. Each one of these two processes is well-defined in the sense of distributions. In consequence, $\<\theta_t^n,G\>$ is well defined.

For a function $G \in \mc S(\bb R)$, let us define $\mc T_0 G:\bb R \to \bb R$ as
\[
\mc T_0 G(x) = \int_{-\infty}^x G(y) dy.
\]
Notice that $\mc T_0 G \in \mc S(\bb R)$ if and only if $\int G(x)dx =0$. Let us define as well $\Lambda: \mc S(\bb R) \to \bb R$ by $\Lambda(G) = \int G(x) dx$. The functional $\Lambda$ is continuous with respect to the topology of $\mc S(\bb R)$. Define $f_0: \bb R \to \bb R$ as $f_0(x) = (1+e^{-x})^{-1}$. The operator $\mc T_0$ does not map  $\mc S(\bb R)$ into itself. We define $\mc T : \mc S(\bb R) \to \mc S(\bb R)$ as
\[
\mc T G = \mc T_0 G - \Lambda(G) f_0,
\]
which now maps $\mc S(\bb R)$ into itself.
We have introduced the operator $\mc T$ in order to relate the interface field $\theta_t^n$ to the density fluctuation field $\mc Y_t^n$. It will be convenient to introduce the incremental interface $\theta_t^{n,*}(x) = \theta_t^n(x)- \theta_0^n(x)$. For a function $G \in \mc C^\infty_c(\bb R)$ such that $\int G(x) dx =0$, it is easy to see that
\[
\<\theta_t^n,G\> = \mc Y_t^n(\mc T G).
\]
In terms of the field $\theta_t^{n,*}$, this relation reads
\begin{equation}
\label{ec19}
\<\theta_t^{n,*}, G\> = \mc Y_t^{n,*}(\mc T G).
\end{equation}

For arbitrary $G \in \mc S(\bb R)$, the following lemma shows how to compute $\<\theta_t^{n,*},G\>$.

\begin{lemma}
\label{l2}
There exists a process $\{\mc Y_t^{n,*}(f_0); t \in [0,T]\}$ such that
\[
\<\theta_t^{n,*}, G\> = \mc Y_t^{n,*}(\mc T G) + \Lambda(G) \mc Y_t^{n,*}(f_0)
\]
for any $G \in \mc S(\bb R)$ and any $t \in [0,T]$.
\end{lemma}

\begin{proof}
Our first task is to prove the existence of the process $\mc Y_t^{n,*}(f_0)$. The idea is to introduce a  cut-off like the one introduced by Rost and Vares. But this time we need to do it in a smooth way. Let us consider the bump function $\zeta: \bb R \to [0,\infty)$ defined by $\zeta(x) = c\exp\{-1/x(1-x)\}$ for $x \in [0,1]$ and $\zeta(x)=0$ otherwise. The constant $c$ is chosen in such a way that $\int g(x) dx =1$. Starting from this bump function we define the interpolating function $g: \bb R \to [0,1]$ as
\[
g(x)=
\begin{cases}
1, & \text{if } x \leq 0\\
1- \int_0^x \zeta(y)dy, &\text{if } 0<x \leq 1\\
0, & \text{if } 1 \leq x.
\end{cases}
\]

Notice that $g'(x) = -\zeta(x)$. We define then the cut-off functions $f_0^l:\bb R \to \bb R$ as $f_0^l(x) = f_0(x)g(x/l)$. We have that $f_0^l \in \mc S(\bb R)$ and 
\[
\frac{d f_0^l}{dx}(x)-f_0'(x) = f_0'(x)\big(g(x/l)-1\big)-\frac{1}{l}f_0(x) \zeta(x/l).
\]
Both terms in this last sum converge to 0 in $L^2(\bb Z)$. Therefore, $\|f_0^l - f_0\|_{1,n} \to 0$ as $l \to \infty$ for any $n$. A more careful computation reveals that the convergence is uniform in $n$. It can also be checked that $\Delta_x^n f_0^l$ converges to $\Delta_x^n f_0$ in $L^2(\bb Z)$ and the convergence is uniform in $n$. From \eqref{ec.M1} we know that
\[
\mc Y_t^{n,*}(f_0^l) = M_t^n(f_0^l) + \mc I_t^n(f_0^l) + \mc A_t^n(f_0^l).
\]

In Section \ref{s2.1} we obtained the estimate $\bb E_n[M_t^n(G)^2] \leq C(\rho) t \|G\|_{1,n}^2$, valid for any function $G \in \mc S(\bb R)$. In particular, the sequence $\{M_t^n(f_0^l);l \in \bb N\}$ is a Cauchy sequence in $L^2(\bb P_n)$. Therefore, for any $t \in [0,T]$ there exists a limit process $M_t^n(f_0)$ such that 
\[
\lim_{l \to \infty} \bb E_n[(M_t^n(f_0^l) -M_t^n(f_0))^2] =0.
\]
By Doob's inequality, the limit is a martingale and we also have convergence at the process level. The same is also true for the processes $\mc I_t^n(f_0^l)$ and $\mc A_t^n(f_0^l)$:
\[
\bb E_n\big[\big(\mc I_t^n(f_0^l) -\mc I_t^n(f_0^m)\big)^2\big] \leq \frac{C(\rho) t^2}{n} \sum_{x \in \bb Z}Ê\big(\Delta_x^n f_0^l -\Delta_x^n f_0^m\big)^2,
\]
\[
\bb E_n\big[\big(\mc A_t^n(f_0^l) -\mc A_t^n(f_0^m)\big)^2\big] \leq \frac{C(\rho) t^{3/2}}{n} \sum_{x \in \bb Z}Ê\big(\nabla_x^n f_0^l -\nabla_x^n f_0^m\big)^2.
\]

Therefore, the limiting processes $\mc I_t^n(f_0)$ and $\mc A_t^n(f_0)$ are also well defined. The convergence at the process level follows from tightness arguments, like in Section \ref{s2.1}. 

At this point we just {\em define} $\mc Y_t^{n,*}(f_0)$ using \eqref{ec.M1}:
\[
\mc Y_t^{n,*}(f_0) = M_t^n(f_0) + \mc I_t^n(f_0) + \mc A_t^n(f_0). 
\]

For $G \in \mc S(\bb R)$ such that $\int G(x)dx =0$, the lemma follows from \eqref{ec19} and an approximation procedure as above. By linearity, we only need to prove the lemma for a single function $G$ such that $\int G(x) dx \neq 0$. For arbitrary functions, the lemma follows from linearity. But we have already done it in the previous section. In fact, at a formal level, what we proved was the identity $\<\theta_t^{n,*},\delta_0\> = \mc Y_t^{n,*}(H_0)$, where $\delta_0$ is the $\delta$ of Dirac at $x=0$ (notice that $\theta_t^{n,*}(0) = \theta_t^n(0)$). A similar argument proves the same identity for the bump function $\zeta$. Notice that we are not able to show convergence at the level of processes for the approximations when $\int G(x) dx \neq 0$, but only in the sense of finite-dimensional distributions. But this is enough to our purposes, since we have already proved that $\mc Y_t^{n,*}(f_0)$ is well defined as a process, and finite-dimensional distributions characterize processes in $\mc D(\mc S'(\bb R), [0,T])$.
\end{proof}

Now we have the elements to prove Theorem \ref{t6}. From \eqref{ec.M1} and Lemma \ref{l2}, we have
\begin{equation}
\label{ec20}
\<\theta_t^{n,*},G\> = \mc I_t^n(\mc T_0 G) + \mc A_t^n(\mc T_0 G) + M_t^n(\mc T_0 G).
\end{equation}

By Proposition \ref{p11}, in order to prove tightness of $\{\theta_t^{n,*}; t \in [0,T]\}; n \in \bb N\}$ it is enough to prove tightness of the projections $\{\<\theta_t^{n,*},G\>; t \in [0,T]\}; n \in \bb N\}$. But this follows as in Section \ref{s2.1}, since the estimates for the processes at the right-hand side of \ref{ec20} only depend on the regularity of $\frac{d}{dx} \mc T_0 G = G$. We have already noticed that $\{\theta_0^n(x); x \in \bb R\}$ converges to a two-sided Brownian motion. Therefore, tightness of $\{\{\theta_t^n; t \in [0,T]\}; n \in \bb N\}$ follows. We are only left with the problem of identifying the limit points of $\theta_t^n$. As in Section \ref{s2.2}, we denote by $n$ a subsequence such that all the processes above converge in distribution to the corresponding limit. 
Notice that $\Delta_x^n \mc T_0 G = \mc T_0 \Delta_x^n G$ for any $G \in \mc S(\bb R)$. Therefore, 
\[
 \int_0^t \mc Y_s^n(\Delta_x^n \mc T_0 G) ds =\int_0^t \mc Y_s^n(\mc T_0 \Delta_x^n G) ds, 
\]
and $\mc I_t^n(\mc T_0 G)$ converges to $\frac{1}{2}\varphi'(\rho)\int_0^t \<\theta_s^n,G''\>ds$ in distribution as $n \to \infty$. Since $\nabla_x^n \mc T_0 G= \mc T_0 \nabla_x^n G$, the process $M_t^n(\mc T_0 G)$ converges to $\mc M_t(G)$, where $\mc M_t$ is a space-time white noise of variance $\chi(\rho)\varphi'(\rho)$. It remains to identify the limit of $\mc A_t^n$ in terms of $\theta_t$. It is enough to write $\mc A_t^{n,\epsilon n}(G)$ in terms of $\theta_t^n$. We have
\[
\mc Y_t^n(i_\epsilon(x)) = \frac{\theta_t^n(x)-\theta_t^n(x+\epsilon)}{\epsilon},
\]
and therefore the limit process $\mc A_t$ satisfies
\[
\mc A_t(\mc T_0 G) = \lim_{\epsilon \to 0} \int_0^t \int_{\bb R}  \int_x^{x+\epsilon} \frac{\beta''(\rho)}{2\epsilon} \Big\{\Big(\frac{\theta_s(x+\epsilon)-\theta_s(x)}{\epsilon}\Big)^2-\frac{\chi(\rho)}{\epsilon}\Big\} G(y) dy dx ds.
\]

Notice the Wick renormalization factor $\chi(\rho)/\epsilon$ in this formula. This factor is now needed because $\int G(x) dx \neq 0$.

Passing to the limit in \eqref{ec20}, we see that  any limit point $\{\theta_t; t \in [0,T]\}$ of the sequence $\{\theta_t^n; t \in [0,T]\}; n \in \bb N\}$ satisfies
\[
\<\theta_t,G\> = \<\theta_0,G\> + \frac{\varphi'(\rho)}{2} \int_0^t \<\theta_s,G''\> ds + \mc B_t(\mc G) + \mc M_t(G),
\]
where $\mc B_t(G) = \mc A_t(\mc T_0 G)$ and $\mc M_t(G)$ is a martingale of quadratic variation $\chi(\rho)\varphi'(\rho)\int G(x)^2 dx$, which proves that $\{\theta_t;t \in [0,T]\}$ is a weak solution of \eqref{ec.KPZ0}. In a similar way, we can prove that $\{\theta_t; t \in [0,T]\}$ is an energy solution of \eqref{ec.KPZ0}.

\appendix

\section{Equivalence of ensembles}
\label{A1}
In this appendix we prove Proposition~\ref{p8}. Let $f: \Omega \to \bb R$ be a local function and assume without loss of generality that there exists $l \in \bb N$ with $\supp(f) \subseteq \{1,\dots,l\}$. Our first observation is that $f$ is a linear, finite combination of functions of the form $\prod_{x \in A} \eta(x)$, where $A \subseteq \{1,\dots,l\}$. Since the thesis of Proposition~\ref{p8} and \ref{p9} are preserved under linear transformations, it is enough to prove the propositions for functions of the form  $\prod_{x \in A} \eta(x)$. We call these functions {\em monomials}. The random variable $\eta^k(0)$ and also the measures $\nu_\rho$, $\nu_{k,m}$ are exchangeable, in the sense that they remain unchanged under a permutation of the random variables $\{\eta(1),\dots,\eta(k)\}$. Therefore, it is enough to prove the propositions for functions of the form $f(\eta)= \eta(1)\cdots\eta(l)$ with $l \in \bb N$. Fix $l \in \bb N$ and take $k \geq l$. Let us recall the definition $\Psi(k,x) = E[f| \eta^k(0) = x]$ for $x$ of the form $m/k$, with $m=0,1,\dots,k$. The conditional expectation $E[f|\eta^k(0)]$ is easy to compute:
\begin{equation}
\label{ec.A2}
E[f|\eta^k(0)=m/k] = \prod_{i=0}^{l-1} \frac{m-i}{k-i} = \prod_{i=0}^{l-1} \frac{k}{k-i} \prod_{i=0}^{l-1} \Big(\frac{m}{k}- \frac{i}{k}\Big).
\end{equation}
Let us call $a_{k,l}$ the first product in the last display. Notice that $a_{k,l}$ is uniformly bounded in $k$, and it converges to $1$ as $k \to \infty$. Developing the second product we have the expansion
\begin{equation}
\label{ec.A1} 
\Psi(k,x)  = a_{k,l} \sum_{i=0}^l \frac{p_i}{k^i}x^{l-i}.
\end{equation}

The coefficients $p_i$ do not depend on $m$ or $k$. Therefore, all the powers of order smaller than $n-1$ in \eqref{ec.A1} are at most of order $1/k^2$, uniformly in $x$ (recall that $0 \leq x \leq 1$).  Therefore, there exists a constant $K_1$ such that 
\[
\sup_{x} \big|\Psi(k,x) - a_{k,l}\big(p_0 x^l + \frac{p_1}{k} x^{l-1}\big)\big| \leq \frac{K_1}{k^2}.
\]
Now we just need to compute $p_0$ and $p_1$. The constant $p_0$ is equal to 1, since in the last product in \eqref{ec.A2} each factor is monic. By the same reason, the coefficient $p_1/k$ is equal to minus the sum of the roots of the aforementioned polynomial. Therefore, $p_1 = -l(l-1)/2$. Up to here we have proved that 
\begin{equation}
\label{ec.A3}
\sup_x \big|\Psi(k,x) - a_{k,l}\big(x^n - \frac{l(l-1)}{2k} x^{n-1}\big)\big| \leq \frac{K_1}{k^2}.
\end{equation}
Now we turn into $a_{k,l}$. It is easier to expand $1/a_{k,l}$. In fact, we have that
\[
\frac{1}{a_{k,l}} = \sum_{i=0}^{l-1}\frac{p_i}{k^i} = 1 - \frac{l(l-1)}{2k} + \frac{r(k)}{k^2},
\]
where $r(k)$ is bounded in $k$. Using the expansion $(1-\epsilon +O(\epsilon^2))^{-1} = 1+\epsilon + O(\epsilon^2)$, we see that
\[
a_{k,l} = 1 +\frac{l(l-1)}{2k} +\frac{\tilde r(k)}{k^2},
\] 
for another function $\tilde r(k)$ bounded in $k$. Putting this asymptotic expansion for $a_{k,l}$ back into  \eqref{ec.A3}, we conclude that there is a constant $c_{eq}$ which only depends on $l$ such that
\[
\sup_x  \Big|\Psi(k,x) - \Big(1+\frac{l(l-1)}{2k}\Big)x^l +\frac{l(l-1)}{2k} x^{l-1}\Big| \leq \frac{c_{eq}}{k^2}.
\]
Remember that for this particular choice of $f$ we have $\varphi(x) = x^l$ and $\varphi''(x) = l(l-1) x^{l-2}$. Replacing above $x^l$ by $\varphi(x)$ and $l(l-1) x^{l-1}$ by $x \varphi''(x)$, Proposition~\ref{p8} is proved.

\section*{Acknowledgements}
P.G. would like to thank the hospitality of Universit\'e Paris-Dauphine, where this work was initiated and of Instituto Nacional de Matem\'atica Pura e Aplicada, IMPA, where this work was finished.  P.G. thanks to CMAT (Portugal) and F.C.T. for finantial support and to F.C.G. for the prize to the research project: "Hydrodynamic Limit of particle systems".
M.J. would like to thank M. Gubinelli for stimulating discussions.

\end{document}